\documentclass[11pt,a4paper,final]{article}

\usepackage{amsmath}    
\usepackage{amsfonts}   
\usepackage{amssymb}    
\usepackage{amsthm}     
\usepackage{mathrsfs}   
\usepackage{graphicx}   

\renewcommand{\baselinestretch}{1.25}
\newlength{\vslength}
\newcommand{\whitespace}{\vspace{\vslength}\par}

\newcommand{\ie}{{\it i.e.}}
\newcommand{\cf}{{\it c.f.}}
\newcommand{\eg}{{\it e.g.}}

\newcommand{\rhs}{{\it r.h.s.}}
\newcommand{\iid}{{\it i.i.d.}}

\newcommand{\GG}{{\mathbb G}}

\newcommand{\PP}{{\mathbb P}}
\newcommand{\RR}{{\mathbb R}}

\newcommand{\scrF}{{\mathscr F}}

\newcommand{\scrL}{{\mathscr L}}

\newcommand{\scrP}{{\mathscr P}}

\newcommand{\score}{{\dot{\ell}}}

\newcommand{\Exp}{\mathrm{Exp}^{+}}
\newcommand{\NExp}{\mathrm{Exp}^{-}}
\newcommand{\1}{\mathbf{1}}
\newcommand{\wt}{\widetilde}

\newcommand{\ep}{{\epsilon}}

\newcommand{\tht}{{\theta}}

\newcommand{\Tht}{{\Theta}}
\newcommand{\sample}{{\underline{X}}}
\renewcommand{\emptyset}{{\varnothing}}

\newcommand{\supp}{\mathop{\rm supp}}

\newcommand{\convas}[1]%
  {{\mathrel{\,\stackrel{#1-a.s.}{\longrightarrow}\,}}}
\newcommand{\convprob}[1]%
  {{\mathrel{\,\stackrel{#1}{\longrightarrow}\,}}}
\newcommand{\convweak}[1]%
  {{\mathrel{\,\stackrel{#1}{\rightsquigarrow}\,}}}

\newcommand{\twobytwo}[4]%
  {\left(\begin{array}{cc} #1 & #2 \\ #3 & #4 \end{array}\right)}
\newcommand{\twovec}[2]%
  {\left({\begin{array}{c} #1\\#2 \end{array}}\right)}

\newtheorem{theorem}{Theorem}[section]

\newtheorem{lemma}{Lemma}[section]

\newtheorem{proposition}{Proposition}[section]

\newtheorem{corollary}{Corollary}[section]

\newtheorem{definition}[theorem]{Definition}

\theoremstyle{remark}

\newtheorem{example}{Example}[section]

\renewenvironment{proof}{\noindent{\bf Proof}\;\;}{\hfill$\Box$\par}

\newcommand{\comment}[1]{{}}
\newcommand{\extra}[1]{{}}

\begin{document}

\thispagestyle{empty}

\title{\vspace*{-9mm}
  Semiparametric posterior limits\\[-1mm]under local asymptotic exponentiality}
\author{Bas~Kleijn$^{1}$~and~Bartek Knapik$^{2}$\\[2mm]
  {\small\it {}$^{1}$\, Korteweg-de~Vries Institute for Mathematics,
    University of Amsterdam}\\
  {\small\it {}$^{2}$\, Department of Mathematics, VU University Amsterdam}
  }
\date{December 2013}
\maketitle

\begin{abstract}
  Consider semiparametric models that display \emph{local asymptotic
  exponentiality} (Ibragimov and Has'minskii (1981) \cite{Ibragimov81}),
  an asymptotic property of the likelihood associated with discontinuities
  of densities. Our interest goes to estimation of the location of
  such discontinuities while other aspects of the density form a nuisance
  parameter. It is shown that under certain conditions on model and prior, 
  the posterior distribution displays Bernstein--von~Mises-type asymptotic
  behaviour, with exponential distributions as the limiting sequence. 
  In contrast to regular settings, the maximum likelihood estimator is 
  inefficient under this form of irregularity. However, Bayesian point 
  estimators based on the limiting posterior distribution attain the minimax risk. 
  Therefore, the limiting behaviour of the posterior is used to advocate 
  efficiency of Bayesian point estimation rather than compare it to 
  frequentist estimation procedures based on the maximum likelihood estimator.
  Results are applied to semiparametric LAE location and scaling examples.\\
  
\end{abstract}


\section{Introduction}
\label{sec:intro}

In recent years, asymptotic efficiency of Bayesian
semiparametric methods has enjoyed much attention. The
general question concerns a non-parametric model $\scrP$ in which
exclusive interest goes to the estimation of a sufficiently smooth,
finite-dimensional functional of interest. Asymptotically, regularity
of the estimator combined with the Cram\'er-Rao bound in the Gaussian
location model that forms the limit experiment \cite{LeCam72}
fixes the rate of convergence to $n^{-1/2}$ and
poses a bound to the accuracy of regular estimators expressed, \eg\
through Haj\'ek's convolution \cite{Hajek70} and asymptotic 
minimax theorems \cite{Hajek72}. In regular Bayesian context, efficiency
of estimation is best captured by a so-called
Bernstein--von~Mises limit (see, \eg\ Le~Cam and Yang (1990)
\cite{LeCam90a}). It should be noted here that efficiency of Bayesian 
estimation in regular models is closely related to asymptotic normality. 
Since the limit is Gaussian, hence symmetric 
and unimodal, the location of the limit, which is any best-regular 
estimator sequence, is directly linked to Bayesian point estimators 
for bowl-shaped loss functions.

Just like frequentist parametric theory for regular
estimates extends quite effortlessly to regular semi-parametric
problems, semi-parametric extensions of Bernstein--von~Mises-type
asymptotic behaviour of posteriors proceeds without essential problems.
Although far from developed fully, some general considerations of
Bayesian semiparametric efficiency are found in
\cite{Bickel12,Boucheron09,Cheng08,Rivoirard12,Shen02} (model- and/or
prior-specific derivations of the Bernstein--von~Mises limit are
many, \eg\ \cite{DeBlasi09,Castillo12,Castillo12b,Johnstone10,deJonge13,
Knapik11,Kruijer12,Leahu11} (of which most are formulated in (the conjugacy
class of) Gaussian white-noise with Gaussian priors)). Limits of
posteriors on sieves are considered in Ghosal (1999, 2000)
\cite{Ghosal99,Ghosal00b} and Bontemps (2011) \cite{Bontemps11}.
Kim and Lee (2004) \cite{Kim04}, Kim (2006, 2009) \cite{Kim06,Kim09}
and, more recently, Castillo and Nickl (2013) \cite{Castillo13} even
consider infinite-dimensional limiting posteriors (notwithstanding the
objections raised in Freedman (1999) \cite{Freedman99}).

However, not all estimators are regular. The quintessential example
calls for estimation of a point of discontinuity of a density: to be
a bit more specific, consider an almost-everywhere differentiable
Lebesgue density on $\RR$ that displays a jump at some point $\tht\in\RR$;
estimators for $\tht$ exist that converge at rate $n^{-1}$ with
exponential limit distributions \cite{Ibragimov81}. To illustrate the
form that this conclusion takes in Bayesian context, consider the
following example. For $\tht\in\RR$, let $F_\tht(x)=(1-e^{-\lambda(x-\tht)})\vee0$, 
where $\lambda > 0$ is fixed and known. Let $X_1, X_2, \ldots$ form an 
\iid\ sample from $F_{\tht_0}$, for some $\tht_0$. It is easy to see that 
the maximum likelihood estimator $\hat\tht_n$ is equal to the minimum 
of the sample $X_{(1)}$. Moreover, $n(X_{(1)}-\tht_0)$ is exponentially 
distributed with rate $\lambda$ for every $n \geq 1$. Therefore, the 
maximum likelihood estimator is consistent and asymptotically 
unbiased with the bias equal to $1/(n\lambda)$. However, 
\begin{equation}
\label{eq:Opt}
	P_0\bigl(n(X_{(1)}-\tht_0)\bigr)^2 = \frac{2}{\lambda^2}, \qquad 
    P_0\Bigl(n\Bigl(X_{(1)}-\frac{1}{n\lambda}-\tht_0\Bigr)\Bigr)^2 = \frac{1}{\lambda^2},
\end{equation}
where $P_0f$ denotes the expectation of a random variable $f$ under $\tht_0$. 
Therefore, the maximum likelihood estimator is inefficient. 

On the other hand consider the following theorem.  
\begin{theorem}
\label{thm:ExpBvM}
Assume that $X_1,X_2,\ldots$ form an \iid\ sample from $F_{\tht_0}$, for some
$\tht_0$. Let $\pi:\RR\rightarrow(0,\infty)$ be a continuous Lebesgue
probability density. Then the associated posterior distribution
satisfies,
\[
  \sup_A \Bigl|\,\Pi_n\bigl(\,\tht \in A \bigm| X_1, \ldots, X_n \,\bigr)
    - \NExp_{X_{(1)},n\lambda}(A)\,\Bigr| \convprob{\tht_0}  0,
\]
where $\NExp_{X_{(1)},n\lambda}$ is a negative exponential distribution with 
rate $n\lambda$ supported on $(-\infty, X_{(1)}]$.
\end{theorem}
The proof of this Bernstein--von~Mises limit is elementary and does
not depend in any crucial way on the particular parametric family of
distributions that we chose (\cf\ also the proof of Theorem~\ref{thm:PostLAE}). 

Consider now the mean $\wt\tht_n = X_{(1)} - 1/(n\lambda)$ of the limiting 
exponential distribution. As seen in (\ref{eq:Opt}) its squared risk is smaller 
than the risk of the maximum likelihood estimator. As a matter of fact, $1/\lambda^2$ 
is the lower bound for the (localized) risk in the exponential experiment. 
This suggests that Bayesian point estimators based on the posterior distribution 
for a wide class of loss functions will be asymptotically minimax.

As a frequentist semi-parametric problem, estimation of a support
boundary point is a well-understood problem (see Ibragimov and
Has'minskii (1981) \cite{Ibragimov81}): assuming that the distribution
$P_{\tht}$ of $X$ is supported on the half-line $[\tht,\infty)$ and an
\iid\ sample $X_1, X_2, \ldots, X_n$ is given, we follow \cite{Ibragimov81}
and estimate $\tht$ with the first order statistic $X_{(1)}=\min_i\{X_i\}$.
If $P_\tht$ has an absolutely continuous Lebesgue density of the form
$p_{\tht}(x)=\eta(x-\tht)\,1\{x\geq\tht\}$, its rate of convergence is
determined by the behaviour of the quantity $\ep\mapsto\int_0^\ep \eta(x)\,dx$
for small values of $\ep$. If $\eta(x) > 0$ for $x$ in a right 
neighbourhood of $0$, then,
\begin{equation*}
  \label{eq:rates}
  n\bigl(\, X_{(1)}-\tht\, \bigr)=O_{P_\tht}(1).
\end{equation*}
For densities of this form, for any sequence $\tht_n$ that converges
to $\tht$ at rate $n^{-1}$, Hellinger distances obey (see
Theorem~VI.1.1 in \cite{Ibragimov81}):
\begin{equation}
  \label{eq:HellingerOpt}
  n^{1/2}\,H(P_{\tht_n},P_{\tht}) = O(1).
\end{equation}
If we substitute the estimators $\tht_n=\hat{\tht}_n(X_1,\ldots,X_n)=X_{(1)}$,
uniform tightness of the sequence in the above display signifies rate
optimality of the estimator (\cf\ Le~Cam (1973, 1986)
\cite{LeCam73,LeCam86}). Regarding asymptotic efficiency beyond
rate-optimality, \eg\ in the sense of minimal asymptotic variance
(or other measures of dispersion of the limit distribution), 
we have already noticed in (\ref{eq:Opt}), in a specific parametric example of 
shifted exponential distributions, that the (one-sided) limit 
distributions one obtains for $X_{(1)}$ can always be improved 
upon by de-biasing (see Section~VI.6, examples~1--3 in \cite{Ibragimov81} 
and Le~Cam (1990) \cite{LeCam90b}).

As a semi-parametric Bayesian question, the matter of estimating
support boundaries is not settled by the above: for the posterior,
it is the local limiting behaviour of the likelihood around the point
of convergence (see, \eg, Theorems~VI.2.1--VI.2.3 in \cite{Ibragimov81})
that determines convergence rather than the behaviour of any
particular statistic. 
The goal of this paper is to shed some light on the
behaviour of marginal posteriors for the parameter of interest
in semi-parametric, irregular estimation problems, through a study
of the Bernstein--von~Mises phenomenon. Only the prototypical
case of a density of bounded variation, supported on the half-line
$[\tht,\infty)$ or on the interval $[0, \tht]$, with a jump at $\tht$, 
is analysed in detail. 
We offer a slight abstraction from the prototypical case, by considering 
the class of models that exhibit a weakly converging expansion of the 
likelihood called \emph{local asymptotic exponentiality (LAE)} 
\cite{Ibragimov81}, to be compared with local asymptotic normality 
\cite{LeCam53} in regular problems. Like in the parametric case of 
Theorem~\ref{thm:ExpBvM}, this type of asymptotic behaviour of the likelihood 
is expected to give rise to a (negative-)exponential marginal posterior 
satisfying the irregular Bernstein--von~Mises limit:
\begin{equation}
  \label{eq:assertion}
  \sup_A \Bigl|\,\Pi_n\bigl(\,h \in A \,\bigm|\, X_1, \ldots, X_n \,\bigr)
    - \NExp_{\Delta_n,\gamma_{\tht_0,\eta_0}}(A)\,\Bigr| \convprob{P_0}  0,
\end{equation}
where $h=n(\tht-\tht_0)$ and the random sequence $\Delta_n$ converges
weakly to exponentiality (see Definition~\ref{df:LAE}). Like argued already 
in the parametric case, the limit (\ref{eq:assertion}) allows for the 
asymptotic identification of Bayesian point estimators based on the 
posterior distribution with the point estimators based on the limiting 
exponential distribution. 
The constant $1/\gamma_{\tht_0,\eta_0}$ determines the scale in the 
limiting exponential distribution and, as such, is related to the asymptotic bound 
for estimators of $\tht$ (for the quadratic loss the bound is exactly given 
by the scale). In this paper, we explore general
sufficient conditions on model and prior to conclude that the limit
(\ref{eq:assertion}) obtains. 

The main theorem is applied in two semi-parametric LAE example models,
one for a shift parameter and one for a scale parameter (compare
with the two regular semiparametric questions in Stein (1956)
\cite{Stein56}). The former one is an extension of the setting considered 
in Theorem~\ref{thm:ExpBvM}, and is closely related to regression problems 
with one-sided errors, often arising in economics.
The later includes a problem of estimation 
of the scale parameter in the family of uniform distributions $[0, \lambda]$, 
($\lambda > 0$). 

\whitespace

The paper is structured as follows: in Section~\ref{sec:Examples} we first 
introduce the notion of local asymptotic exponentiality and then present
two semiparametric LAE models satisfying the exponential Bernstein--von~Mises 
property (\ref{eq:assertion}) asymptotically. In Section~\ref{sec:main} we 
give the main theorem and a corollary
that simplifies the formulation. In Section~\ref{sec:ape}, the proof
of the main theorem is built up in several steps, from a particular 
type of posterior convergence, to an LAE expansion for
integrated likelihoods and on to posterior exponentiality of the type
described by (\ref{eq:assertion}). Section~\ref{sec:proofs} contains 
the proofs of auxiliary results needed in the proof of the main theorem, as well as 
verification of the conditions of the simplified corollary for 
the two models presented in Section~\ref{sec:Examples}. 

\subsection*{Notation and conventions}
The (frequentist) true distribution of each of the data points
in the \iid\ sample $\sample_n=(X_1,\ldots,X_n)$ is denoted $P_0$
and assumed to lie in the model $\scrP$. Associated order statistics are
denoted $X_{(1)},X_{(2)},\ldots$. The location-scale family
associated with the exponential distribution is denoted $\Exp_{\Delta,\lambda}$
and its negative version by $\NExp_{\Delta,\lambda}$. We localise $\tht$
by introducing $h = n(\tht-\tht_0)$ with inverse $\tht_n(h)=\tht_0+n^{-1}h$.
The expectation of a random variable $f$ with respect to a probability
measure $P$ is denoted $Pf$; the sample average of $g(X)$ is denoted
$\PP_ng(X) = (1/n)\sum_{i=1}^ng(X_i)$ and 
$\GG_ng(X) = n^{1/2}(\PP_ng(X)-Pg(X))$. If $h_n$ is stochastic,
$P_{\tht_n(h_n),\eta}^nf$ denotes the integral
$\int f(\omega)\,(dP_{\tht_n(h_n(\omega)),\eta}^n/dP_0^n(\omega))
(\omega)\,dP_0^n(\omega)$. The Hellinger distance between $P$ and
$P'$ is denoted $H(P,P')$ and induces a metric $d_H$ on the space
of nuisance parameters $H$ by $d_H(\eta,\eta')=
H(P_{\tht_0,\eta},P_{\tht_0,\eta'})$, for all $\eta,\eta'\in H$.
A prior on (a subset $\Tht$ of) $\RR^k$ is said to be
\emph{thick} (at $\tht\in\Tht$) if it is Lebesgue absolutely continuous
with a density that is continuous and strictly positive (at $\tht$).


\section{Local asymptotic exponentiality and estimation of support boundary points}
\label{sec:Examples}

Throughout this paper we consider estimation of a functional
$\theta: \scrP \to \RR$ on a nonparametric model $\scrP$ based
on a sample $X_1,X_2,\ldots$, distributed \iid\ according to some
unknown $P_0\in\scrP$. We assume that $\scrP$ is parametrized in
terms of a one-dimensional parameter of interest $\tht\in\Tht$ and
a nuisance parameter $\eta\in H$ so that we can write
$\scrP=\{P_{\theta,\eta}:\tht\in\Tht,\eta\in H\}$, and that
$\scrP$ is dominated by a $\sigma$-finite measure on the sample space
with densities $p_{\tht,\eta}$. 
The set $\Theta$ is open in $\RR$, and $(H,d_H)$ is an
infinite dimensional metric space (to be specified further at later
stages). Assuming identifiability, there exist unique
$(\theta_0,\eta_0)\in\Tht\times H$ such that $P_0=P_{\theta_0,\eta_0}$.
Assuming measurability of the map $(\theta,\eta) \mapsto P_{\theta,\eta}$
and priors $\Pi_\Tht$ on $\Tht$ and $\Pi_H$ on $H$, the prior $\Pi$
on $\scrP$ is defined as the product prior $\Pi_\Tht\times\Pi_H$ on
$\Tht\times H$, lifted to $\scrP$. The subsequent sequence of posteriors
\cite{Ghosal00a} takes the form,
\begin{equation}
  \label{eq:posterior}
  \Pi_n\bigl( A \vert X_1, \ldots, X_n\bigr)
  = \int_A \prod_{i=1}^n p(X_i)\, d\Pi(P) \Bigg/
    \int_\scrP \prod_{i=1}^n p(X_i)\, d\Pi(P),
\end{equation}
where $A$ is any measurable model subset.

Throughout most of this paper, the parameter of interest $\tht$ is
represented in localised form, by centering on $\theta_0$ and
rescaling: $h = n(\theta-\theta_0)\in\RR$. (We also make use of
the inverse $\theta_n(h) = \theta_0 + n^{-1}h$.) The following
(irregular) local expansion of the likelihood is due to Ibragimov and
Has'minskii (1981) \cite{Ibragimov81}.
\begin{definition}[Local asymptotic exponentiality]
\label{df:LAE}
A one-dimensional parametric model $\tht\mapsto P_{\tht}$ is said to
be \emph{locally asymptotically exponential} (LAE) at $\tht_0\in\Tht$
if there exists a sequence of random variables $(\Delta_n)$ and
a positive constant $\gamma_{\tht_0}$ such that for all $(h_n)$,
$h_n\rightarrow h$,
\[
  \prod_{i=1}^n \frac{p_{\tht_0 + n^{-1}h_n}}{p_{\tht_0}}(X_i)
  = \exp(h\gamma_{\tht_0} + o_{P_{\tht_0}}(1))\,\1_{\{h\leq \Delta_n\}},
\]
with $\Delta_n$ converging weakly to $\Exp_{0,\gamma_{\tht_0}}$.
\end{definition}

In many examples, \eg\ that of Subsection~\ref{sub:semishift},
$\Delta_n$ and its weak limit are independent of $\tht_0$.
This definition should be viewed as an irregular variation on
the one-dimensional version of Le~Cam's local asymptotic normality (LAN)
\cite{LeCam53}, which forms the smoothness requirement in the context
of the Bernstein--von~Mises theorem (see, \eg\ van der Vaart (1998) \cite{vdVaart98}). 
Therefore, an LAE expansion is expected to give rise to a one-sided,
exponential marginal posterior limit, \cf\ (\ref{eq:assertion}). 

In the main result of the paper we use a slightly stronger version of 
local asymptotic exponentiality. We say that the model is {\em stochastically}
LAE if the LAE property holds for every random sequence $(h_n)$ 
that is bounded in probability. Therefore, $h$ in the expansion is also 
replaced with $h_n$. 
\whitespace

We now turn to two examples of support boundary estimation
for which the likelihood displays an LAE expansion. In
Subsection~\ref{sub:semishift} the parameter of interest is a shift
parameter, while in Subsection~\ref{sub:semiscale} we consider a
semiparametric scaling family.

\subsection{Semiparametric shifts}
\label{sub:semishift}

The so-called \emph{location problem} is one of the classical problems
in statistical inference: let $X_1, X_2, \ldots$ be \iid\ real-valued
random variables, each with marginal $F_\mu:\RR\rightarrow[0,1]$, where
$\mu\in\RR$ is the \emph{location}, \ie\ the distribution function $F_\mu$
is some fixed distribution $F$ shifted over $\mu$: $F_\mu(x)=F(x-\mu)$.

Depending on the nature of $F$, the corresponding location estimation
problem can take various forms: for instance, in case $F$ possesses a
density $f:\RR\rightarrow[0,\infty)$ that is symmetric around $0$ (and
satisfies the regularity condition $\int (f'/f)^2(x)\,dF(x)<\infty$),
the location $\mu$ is estimated at rate $n^{-1/2}$ (equally well whether
we know $f$ or not \cite{Stein56}). If $F$ has a support that is contained
in a half-line in $\RR$ (\ie\ if there is a domain boundary), the problem
of estimating the location might become easier, as noticed in the 
example given in the introduction.

The problem of estimating the boundary of a distribution has important 
practical motivations, arising in certain auction models, search models, 
production frontier models, as well as truncated- or censored-regression models. 
For instance, assume the data $(X_1, Y_1), (X_2, Y_2), \ldots$ are generated by the 
model
\[
	Y_i = f(X_i)+\mu + \ep_i, 
\]
where $f$ denotes a smooth function satisfying $f(0)=0$, and for 
simplicity both $X_i$ and $Y_i$ are scalars. Moreover, we suppose that the 
density of the error $\ep_i$, conditional on $X_i = x$, is supported 
on $[0, \infty)$. Therefore, the quantities $f$ and  $\mu$ represent a boundary, 
and we are interested in a certain aspect of it, namely $\mu$ itself.  
For more details and more examples we refer the reader to 
Hirano and Porter (2003) \cite{Hirano03}, 
Chernozhukov and Hong (2004) \cite{Chernozhukov04}, 
Hall and van Keilegom (2009) \cite{Hall09}. 

In this subsection we consider a model of densities with a 
discontinuity at $\mu$: we assume that
$p(x) = 0$ for $x<\mu$ and $p(\mu)>0$ while $p:\RR\to[0,\infty)$ is
continuous at all $x\geq\mu$. Observed is an \iid\ sample $X_1, X_2,
\ldots$ with marginal $P_0$. The distribution $P_0$ is assumed to have 
a density of above form, \ie\ with unknown location $\tht$ for a
nuisance density $\eta$ in some space $H$. Model distributions
$P_{\tht,\eta}$ are then described by densities,
\[
  p_{\tht,\eta}:[\theta,\infty)\rightarrow[0,\infty): x\mapsto\eta(x-\tht),
\]
for $\eta\in H$ and $\tht\in\Tht\subset\RR$. As for the family $H$ of
nuisance densities, our interest does not lie in modelling of the
tail, we concentrate on specifying the behaviour at the
discontinuity.
For that reason (and in order to connect with
Theorem~\ref{thm:LAESBvM}), we impose some conditions on the nuisance
space $H$: assume that
$\eta:[0,\infty)\rightarrow[0,\infty)$ is differentiable and that
$\score(t) = \eta'(t)/\eta(t)+\alpha$ is a bounded
continuous function with a limit at infinity. For given $S>0$, let
$\scrL$ denote the ball of radius $S$ in the space
$(C[0,\infty],\|\cdot\|_{\infty})$ of continuous functions
from the extended half-line to $\RR$ with uniform norm.
An Esscher transform of the form 
\begin{equation}
  \label{eq:Esscher}
  \eta_\score(x)
    =\frac{e^{-\alpha\,x+\int_0^x\score(t)\,dt}}
   {\int_0^\infty e^{-\alpha\,y+\int_0^y\score(t)\,dt}\,dy},
\end{equation}
for $x\geq0$, maps $\scrL$ to the space $H$ which we choose to
model the nuisance.

Properties of this mapping (\cf\ Lemma~\ref{lem:Esscher}) 
guarantee that $H$ consists of functions of bounded variation, 
hence Theorem~V.2.2 in Ibragimov and Has'minskii (1981) \cite{Ibragimov81}
confirms that the model exhibits local asymptotic exponentiality
in the $\theta$-direction for every fixed $\eta$. In the notation of
Definition~\ref{df:LAE}, $\gamma_{\tht_0,\eta} = \eta(0)$, \ie\
the size of the discontinuity at zero. Since it is not difficult to
find a prior on a space of bounded continuous functions (see, \eg\
Lemma~\ref{lem:Prior} below), (Borel) measurability of the Esscher
transform as a map between $\scrL$ and $H$ enables a push-forward
prior on $H$.
\begin{theorem}
\label{thm:SBvM-Example}
Let $X_1, X_2, \ldots$ be an \iid\ sample from the location model
introduced above with $P_0 = P_{\theta_0,\eta_0}$ for some
$\theta_0 \in \Theta$, $\eta_0\in H$. Endow $\Theta$ with a prior that
is thick at $\theta_0$ and $\scrL$ with a prior $\Pi_\scrL$ such
that $\scrL\subset\supp(\Pi_\scrL)$. Then the marginal posterior
for $\theta$ satisfies,
\begin{equation}
  \label{eq:Assertion}
  \sup_A\Bigl|\Pi(n(\theta-\theta_0)\in A\, |\, X_1, \ldots, X_n)
    - \NExp_{\Delta_n,\gamma_{\tht_0,\eta_0}}(A)\Bigr|\convprob{P_0} 0,
\end{equation}
where $\Delta_n$ is exponentially distributed with rate
$\gamma_{\tht_0,\eta_0} = \eta_0(0)$.
\end{theorem}

Details of the proof of Theorem~\ref{thm:SBvM-Example} can 
be found in Subsection~\ref{sub:proofssemishift}.

\subsection{Semiparametric scaling}
\label{sub:semiscale}

Another important statistical problem is related to the \textit{scale} 
or \textit{dispersion} of the probability distribution: let 
$X_1, X_2, \ldots$ be \iid\ real-valued random variables, each with marginal 
$F_\lambda:\RR\rightarrow[0,1]$, where $\lambda\in(0,\infty)$ is the 
\emph{scale}, \ie\ the distribution function $F_\lambda$ is some fixed 
distribution $F$ scaled by $\lambda$: $F_\lambda(x)=F(x/\lambda)$.

Again, depending on the nature of $F$, the corresponding scale estimation
problem can take various forms: for instance, in case $F$ possesses a
density $f:\RR\rightarrow[0,\infty)$ with support $\RR$ that is 
absolutely continuous (and satisfies the regularity condition
$\int (1+x^2)(f'/f)^2(x)\,dF(x)<\infty$), the scale $\lambda$ is estimated
at rate $n^{-1/2}$ (equally well whether we know $f$ or not, as conjectured
in \cite{Stein56}, and studied later in \cite{Wolfowitz74} and
\cite{Park90}). If $F$ is supported on $[0,\infty)$ (or $(-\infty, 0])$, 
the problem can be reparametrized and viewed as a regular location problem.
When $F$ has a support that is a closed interval with one non-zero
endpoint (\ie\ only one point of the support varies with scale), the
problem of estimating the scale might become easier. Probably the best
known example of this type is estimation of the scale parameter in the
family of the uniform distributions $[0,\lambda]$, $(\lambda>0)$.

In this subsection we consider an extension of this uniform example: 
we assume that $p(x) > 0$ for $x\in[0,\lambda]$ and $0$ otherwise while 
$p:[0,\lambda]\to[0,\infty)$ is continuous at all $x\in(0,\lambda)$. 
Observed is an \iid\ sample $X_1, X_2, \ldots$ with marginal $P_0$. 
The distribution $P_0$ is assumed to have a density of above form, \ie\ with 
unknown scale $\tht$ for a nuisance density $\eta$ in some space $H$. Model 
distributions $P_{\tht,\eta}$ are then described by densities,
\begin{equation}
  \label{eq:scaledens}
  p_{\tht,\eta}:[0,\theta]\rightarrow[0,\infty): 
    x\mapsto\frac{1}{\theta}\,\eta\Bigl(\frac{x}{\tht}\Bigr),
\end{equation}
for $\eta\in H$ and $\tht\in\Tht\subset(0,\infty)$. Fix $S > 0$ and assume 
that $\eta:[0,1]\rightarrow[0,\infty)$ is monotone increasing, differentiable 
and bounded, and that $\score(t) = \eta'(t)/\eta(t) - S$ is a bounded 
continuous function. For given $S>0$, let $\scrL$ denote the ball of 
radius $S$ in the normed space $(C[0,1],\|\cdot\|_{\infty})$ of continuous 
functions from the unit interval to $\RR$ with uniform norm. The following Esscher 
transform maps $\scrL$ to the space $H$ with which we choose to model the nuisance:
\begin{equation}
  \label{eq:Esscher2}
  \eta_\score(x)
    =\frac{e^{Sx+\int_0^x\score(t)\,dt}}
   {\int_0^1 e^{Sy+\int_0^y\score(t)\,dt}\,dy},
\end{equation}
for $x\in[0,1]$.

Theorem~V.2.2 in \cite{Ibragimov81} verifies local asymptotic 
exponentiality in the $\theta$-direction for every fixed $\eta$, although 
in its positive version. This does not pose problems in applying results
of this paper: we maintain the sign for $h$ and write
$\Delta_n = -\nabla_n$, where $\nabla_n = n(\tht_0-X_{(n)})$. 
In the notation of Definition~\ref{df:LAE}, $\gamma_{\tht_0,\eta}
= \eta(1)/\tht_0$, \ie\ the rate of the limiting exponential distribution
is the size of the discontinuity at the varying endpoint of the support. 
Again, we use a push-forward prior on $H$ based on a prior for $\scrL$.

As already noted, our scaling and location problems are both
LAE and the parametrizations and solutions we formulate are closely related.
However, the nuisance parametrizations are quite different and the
relation between the models is a
subtle one. Therefore the location theorem of the previous subsection
and the scaling theorem that follows are very similar in appearance,
but form the answers to quite distinct questions.
\begin{theorem}
\label{thm:SBvM-Example2}
Let $X_1, X_2, \ldots$ be an \iid\ sample from the scale model
introduced above with $P_0 = P_{\theta_0,\eta_0}$ for some
$\theta_0 \in \Theta$, $\eta_0\in H$. Endow $\Theta$ with a prior that
is thick at $\theta_0$, and $\scrL$ with a prior 
$\Pi_\scrL$ such that $\scrL\subset\supp(\Pi_\scrL)$. Then the marginal 
posterior for $\theta$ satisfies,
\begin{equation}
  \label{eq:Assertion2}
  \sup_A\Bigl|\Pi(n(\theta-\theta_0)\in A\, |\, X_1, \ldots, X_n)
    - \Exp_{-\nabla_n,\gamma_{\tht_0,\eta_0}}(A)\Bigr|\convprob{P_0} 0,
\end{equation}
where $\nabla_n$ is exponentially distributed with rate
$\gamma_{\tht_0,\eta_0} = \eta_0(1)/\theta_0$.
\end{theorem}

Details of the proof of Theorem~\ref{thm:SBvM-Example2} can 
be found in Subsection~\ref{sub:proofssemiscale}.


\section{General results}
\label{sec:main}

In order to establish the limit (\ref{eq:assertion}) (also (\ref{eq:Assertion}) 
and (\ref{eq:Assertion2})), we study posterior
convergence of a particular type, termed \emph{consistency under perturbation}
in \cite{Bickel12}. One can compare this type of consistency with ordinary
posterior consistency in non-parametric models, except here the
non-parametric component is the nuisance parameter $\eta$ and we
allow for (stochastic) perturbation by (local) deformations of the
parameter of interest $\tht_n(h_n)=\tht_0+n^{-1}h_n$. In regular situations, 
this gives rise to accumulation of posterior mass around so-called
least-favourable submodels, but here the parameter
of interest is irregular and the situation is less involved: accumulation
of posterior mass occurs around $(\tht_n(h_n),\eta_0)$. Therefore,
posterior consistency under perturbation describes concentration in
$d_H$-neighbourhoods of the form, $(\rho>0)$,
\begin{equation}
\label{eq:H-ball}
  D(\rho) = \{\eta \in H: d_H(\eta,\eta_0) < \rho\}.
\end{equation}
To guarantee sufficiency of prior mass around the point of convergence,
we use Kullback--Leibler-type neighbourhoods of the form,
\begin{equation}
\label{eq:KLngb}
\begin{split}
  K_n(\rho, M) = \bigg\{\eta \in H: P_0\bigg(&\sup_{|h|\leq M}
    - \1_{A_{\theta_n(h),\eta}}\log \frac{p_{\theta_n(h),\eta}}
         {p_{\theta_0,\eta_0}}\bigg)\leq \rho^2,\\
    P_0\bigg(&\sup_{|h|\leq M} - \1_{A_{\theta_n(h),\eta}}
         \log \frac{p_{\theta_n(h),\eta}}{p_{\theta_0,\eta_0}}
      \bigg)^2\leq \rho^2 \bigg\},
\end{split}
\end{equation}
where, in the present LAE setting, 
\[
  A_{\theta_n(h),\eta} =
  \bigg\{x: \frac{p_{\theta_n(h),\eta}}{p_{\theta_0,\eta_0}}(x) > 0\bigg\}.
\]
Note that 
$\prod_{i=1}^n \1_{A_{\theta_n(h),\eta}}(X_i) = \1_{\{h \leq \Delta_n\}}$,
as in the LAE expansion.

Suppose that $A$ in (\ref{eq:posterior}) is of the form $A = B\times H$ for
some measurable $B \subset \Theta$. Since we use a product prior
$\Pi_\Theta \times \Pi_H$, the marginal posterior of the parameter
$\theta \in \Theta$ depends on the nuisance factor only through
the integrated likelihood,
\begin{equation}
\label{eq:IL}
  S_n : \Theta \to \RR : \theta \mapsto 
  \int_H\prod_{i=1}^n\frac{p_{\theta,\eta}}{p_{\theta_0,\eta_0}}(X_i)\,d\Pi_H(\eta),
\end{equation}
and its localised version, $h\mapsto s_n(h) = S_n(\tht_0+n^{-1}h)$.
One of the conditions of the subsequent theorem is a domination condition
based on the quantities,
\[
  U_n(\rho, h_n) = \sup_{\eta \in D(\rho)}
    P^n_{\theta_0,\eta} \bigg(\prod_{i=1}^n
    \frac{p_{\theta_n(h_n),\eta}}{p_{\theta_0,\eta}}(X_i)\bigg),
\]
Another condition required in the irregular version of the semiparametric
Bernstein--von Mises theorem is one-sided contiguity (\cf\ condition
{\it (iv)} of Theorem~\ref{thm:LAESBvM} below). Lemma~\ref{lem:DomLemma}
shows that such one-sided contiguity and domination as in (\ref{eq:domination})
are closely related and provides two different sufficient conditions
for both to hold in general. The log-Lipschitz construction is used 
in the examples of Section~\ref{sec:Examples}; in other applications
of the theorem it may be more convenient to by-pass 
Lemma~\ref{lem:DomLemma} and prove (\ref{eq:domination}) and contiguity
directly from the model definition.
\begin{theorem}[Irregular semiparametric Bernstein--von Mises]
\label{thm:LAESBvM}
Let $X_1,X_2,\ldots$ be distributed \iid-$P_0$, with $P_0\in\scrP$. Let
$\Pi_H$ and $\Pi_\Tht$ be priors on $H$ and $\Tht$ and assume that
$\Pi_\Theta$ is thick at $\theta_0$. Suppose that $\theta \mapsto 
P_{\theta,\eta}$ is stochastically LAE in the $\theta$-direction, for 
all $\eta$ in a $d_H$-neighbourhood of $\eta_0$ and that 
$\gamma_{\tht_0,\eta_0}>0$. Assume also that for large enough
$n$, the map $h \mapsto s_n(h)$ is continuous on $(-\infty,\Delta_n]$,
$P_0^n$-almost-surely. Furthermore, assume that there exists a sequence 
$(\rho_n)$ with $\rho_n \downarrow 0$, $n\rho_n^2 \to \infty$ such that,
\begin{itemize}
  \item[(i)] for all $M>0$, there exists a $K>0$ such that for large enough $n$,
    \[
      \Pi_H\bigl(K_n(\rho_n, M)\bigr) \geq e^{-Kn\rho_n^2},
    \]
  \item[(ii)] for all $n$ large enough, the Hellinger metric entropy satisfies,
    \[
      N\bigl(\rho_n, H, d_H\bigr) \leq e^{n\rho_n^2},
    \]
\end{itemize}
and, for every bounded, stochastic $(h_n)$,
\begin{itemize}
  \item[(iii)] the model satisfies the domination condition,
    \begin{equation}
    \label{eq:domination}
      U_n(\rho_n, h_n) = O(1),
    \end{equation}
  \item[(iv)] for every $\eta \in D(\rho)$ for $\rho > 0$ small
    enough, the sequence $P^n_{\theta_n(h_n),\eta}$ is contiguous with
    respect to the sequence $P^n_{\theta_0,\eta}$,
  \item[(v)] and for all $L > 0$, Hellinger distances satisfy the
    uniform bound,
    \[
      \sup_{\eta\in D^c(L\rho_n)}
        \frac{H\bigl(P_{\theta_n(h_n),\eta}, P_{\theta_0, \eta}\bigr)}
             {H\bigl(P_{\theta_0,\eta},P_0\bigr)} = o(1).
    \]
\end{itemize}
Finally, suppose that,
\begin{itemize}
  \item[(vi)] for every $(M_n)$, $M_n \to \infty$, the posterior satisfies
    \[
      \Pi_n\bigl(|h|\leq M_n \big| X_1, \ldots, X_n \bigr) \convprob{P_0}1.
    \]
\end{itemize}
Then the sequence of marginal posteriors for $\theta$ converges in
total variation to a negative exponential distribution,
\begin{equation}\label{eq:TVBvM}
  \sup_A \Bigl| \Pi_n\bigl(h \in A \big| X_1, \ldots, X_n \bigr)
    - \NExp_{\Delta_n, \gamma_{\tht_0,\eta_0}}(A)\Bigr| \convprob{P_0} 0.
\end{equation}
\end{theorem}
Regarding the nuisance rate of convergence $\rho_n$, conditions {\it (i)}
and {\it (ii)} are expected in some form or other in order to achieve
consistency under perturbation. As stated, they almost coincide
with requirements for non-parametric convergence at rate
$(\rho_n)$ without a parameter of interest \cite{Ghosal00a}. A simplified
version of Theorem~\ref{thm:LAESBvM} that does not refer to any specific
nuisance $\rho_n$ is stated as Corollary~\ref{cor:RateFree}. In the
rate-free case of Corollary~\ref{cor:RateFree}, conditions on prior mass
and entropy numbers ({\it (i)} and {\it (ii)}) essentially require
nuisance consistency (at some rate rather than a specific one), thus
weakening requirements on model and prior. Concerning conditions
{\it (iii)--(v)}, note that, typically, the numerator in condition~{\it (v)}
converges to zero at rate $O(n^{-1/2})$, \cf\  (\ref{eq:HellingerOpt}),
while the denominator goes to zero at slower, non-parametric rate. As such, 
condition~{\it (v)} is to be viewed as a weak condition that rarely poses
a true restriction on the applicability of the theorem. Furthermore,
Lemma~\ref{lem:DomLemma} formulates two slightly stronger conditions to
validate both {\it (iii)} and {\it (iv)} above for any rate $(\rho_n)$. 

Condition {\it (vi)} of Theorem~\ref{thm:LAESBvM} appears to be the
hardest to verify in applications. On the other hand
it cannot be weakened since {\it (vi)} also follows from \eqref{eq:TVBvM}.
Besides condition {\it (i)}, only condition {\it (vi)} implies
a requirement on the nuisance prior $\Pi_H$. 
Experience with the examples of Section~\ref{sec:Examples} 
suggests that conditions {\it (i)}--{\it (v)}
are relatively weak in applications, while {\it (vi)} harbours the
potential for negative surprises, mainly due to semiparametric bias
leading to sub-optimal asymptotic variance, sub-optimal marginal rate
or even marginal inconsistency. On the other hand, there are conditions
under which condition~{\it (vi)} is easily seen to be valid: in
Section~\ref{sub:Marginal} we present a model condition that guarantees
marginal posterior convergence according to {\it (vi)} for \emph{any}
choice of the nuisance prior $\Pi_H$.

As discussed already after Theorem~\ref{thm:LAESBvM}, in many situations
the domination condition holds for \emph{any rate} $(\rho_n)$. This
circumstance simplifies the result substantially, leading to the
conditions that are comparable to those of Schwartz' consistency
theorem (see Schwartz (1965) \cite{Schwartz65}).
\begin{corollary}[Rate-free irregular semiparametric Bernstein--von Mises]
\label{cor:RateFree}
Let $X_1, X_2, \ldots$ be distributed \iid-$P_0$, with $P_0 \in \scrP$
and let $\Pi_\Theta$ be thick at $\theta_0$. Suppose that $\theta \mapsto 
P_{\theta, \eta}$ is stochastically LAE in the $\theta$-direction, 
for all $\eta$ in a $d_H$-neighbourhood of $\eta_0$ and that 
$\gamma_{\tht_0,\eta_0}$ is strictly positive. Also assume that
for large enough $n$, the map $h \mapsto s_n(h)$ is continuous on
$(-\infty,\Delta_n]$ $P^n_0$-almost-surely.  Furthermore,
assume that,
\begin{itemize}
  \item[(i)] for all $\rho > 0$, the Hellinger metric entropy
    satisfies $N(\rho, H, d_H) <\infty$, and the nuisance prior
    satisfies $\Pi_H(K(\rho))>0$,
  \item[(ii)] for every $M > 0$, there exists an $L > 0$ such
    that for all $\rho > 0$ and large enough $n$ $K(\rho) \subset
    K_n(L\rho , M)$,
\end{itemize}
and that for every bounded, stochastic $(h_n)$,
\begin{itemize}
  \item[(iii)] there exists an $r > 0$ such that $U_n(r, h_n)= O(1)$, 
  \item[(iv)] for every $\eta \in D(r)$ the sequence
    $P^n_{\theta_n(h_n),\eta}$ is contiguous to the sequence
    $P^n_{\theta_0,\eta}$,
  \item[(v)] and that Hellinger distances satisfy,
    $\sup_{\eta \in H} H(P_{\theta_n(h_n),\eta}, P_{\theta_0, \eta}) = O(n^{-1/2})$.
\end{itemize}
Finally, assume that,
\begin{itemize}
  \item[(vi)] for every $(M_n)$, $M_n \to \infty$, the posterior satisfies,
    \[
      \Pi_n\bigl(|h|\leq M_n | X_1, \ldots, X_n \bigr) \convprob{P_0} 1.
    \] 
\end{itemize}
Then marginal posteriors for $\theta$ converge in total variation to
a negative exponential distribution,
\[
  \sup_A \Bigl| \Pi_n\bigl(h \in A | X_1, \ldots, X_n\bigr)
    - \NExp_{\Delta_n,\gamma_{\tht_0,\eta_0}}(A)\Bigr| \convprob{P_0} 0.
\]
\end{corollary}
\begin{proof}
Under conditions {\it (i)}, {\it (ii)}, {\it (v)}, and the stochastic
LAE assumption, the assertion of Corollary~\ref{cor:PertRateFree} holds.
Due to conditions {\it (iii)} (and {\it (iv)}), conditions
{\it (iii)} (respectively {\it (iv)}) in Theorem~\ref{thm:LAESBvM} are
satisfied for large enough $n$. Condition {\it (vi)} then suffices
for the assertion of Theorem~\ref{thm:PostLAE}.
\end{proof}


\section{Asymptotic posterior exponentiality}
\label{sec:ape}

In this section we give the proof of Theorem~\ref{thm:LAESBvM} in
several steps: the first step (Subsection~\ref{sub:pert}) is
a proof of consistency under perturbation under a
condition on the nuisance prior $\Pi_H$ and a testing condition. In
Subsection~\ref{sub:MPAE} we show that the \emph{integral} of the
likelihood with respect to the nuisance prior displays an LAE-expansion,
if consistency under perturbation obtains and contiguity/domination
conditions are satisfied. In the third step, also discussed in
Subsection~\ref{sub:MPAE}, we show that an LAE-expansion
of the integrated likelihood gives rise to a semiparametric exponential
limit for the posterior in total variation, if the marginal posterior
for the parameter of interest converges at $n^{-1}$-rate. The
rate of marginal convergence depends on the control of likelihood 
ratios, which is discussed in Subsection~\ref{sub:Marginal}. 
Put together, the results constitute a proof of Theorem~\ref{thm:LAESBvM}. 
Stated conditions are verified in Section~\ref{sec:proofs} for the 
two examples of Section~\ref{sec:Examples}.

\subsection{Posterior convergence under perturbation}
\label{sub:pert}

Given a rate sequence $(\rho_n)$, $\rho_n \downarrow 0$, we say that the
conditioned nuisance posterior is \emph{consistent under
$n^{-1}$-perturbation at rate $\rho_n$}, if, for all bounded, stochastic
sequences $(h_n)$,
\[
  \Pi_n\bigl(\,D^c(\rho_n)\bigm|\theta=\theta_0+n^{-1}h_n,X_1,\ldots,X_n\,\bigr)
  \convprob{P_0} 0,
\]
For a more elaborate discussion of this property, the reader is
referred to Bickel and Kleijn (2012) \cite{Bickel12}.
\begin{theorem}[Posterior convergence under perturbation]
\label{thm:PertRate}
Assume there is a sequence $(\rho_n)$, $\rho_n\downarrow0$,
$n\rho_n^2\rightarrow\infty$ with the property that for all $M>0$ there
exist a $K>0$ such that,
\[
  \Pi_H(K_n(\rho_n,M)) \geq e^{-Kn\rho_n^2},\quad
  N\bigl(\rho_n,H,d_H\bigr)\leq e^{n\rho_n^2},
\]
for large enough $n$. Assume also that for all $L>0$ and all bounded,
stochastic $(h_n)$,
  \begin{equation}
    \label{eq:Hcone}
    \sup_{\eta\in D^c(L\rho_n)}\,
      \frac{H(P_{\tht_n(h_n),\eta},P_{\tht_0,\eta})}{H(P_{\tht_0,\eta},P_0)}=o(1).
  \end{equation}
Then, for every bounded, stochastic $(h_n)$ there exists an $L>0$ such that,
\[
  \Pi_n\bigl(\,D^c(L\rho_n)\,\bigm|\,
    \theta=\theta_0+n^{-1}h_n,X_1,\ldots,X_n\,\bigr)=o_{P_0}(1).
\]
\end{theorem}
The proof of this theorem can be broken down into two separate steps,
with the following testing condition in between: for every bounded,
stochastic $(h_n)$ and all $L>0$ large enough, a test sequence
$(\phi_n)$ and constant $C>0$ must exist, such that,
\begin{equation}
  \label{eq:test-conda}
  P^n_0\phi_n\rightarrow0,\qquad
  \sup_{\eta\in D^c(L\rho_n)}P^n_{\theta_n(h_n),\eta}(1-\phi_n)\leq e^{-CL^2n\rho_n^2},
\end{equation}
for large enough $n$. According to Lemma~3.2 in \cite{Bickel12}, the
metric entropy condition and ``cone condition'' (\ref{eq:Hcone}) suffice
for the existence of such a test sequence. 
While the above testing argument
is instrumental in the control of the numerator of (\ref{eq:posterior}),
the denominator of the posterior is lower-bounded with the help of the
following lemma, which adapts Lemma~8.1 in \cite{Ghosal00a} to
$n^{-1}$-perturbed, irregular setting. The proof of Theorem~\ref{thm:PertRate} 
follows then the proof of Theorem~3.1 in \cite{Bickel12}.
\begin{lemma}
\label{lem:GGVlike}
Let $(h_n)$ be stochastic and bounded by some $M>0$. Then
\[
  P^n_0\bigg(
    \bigg\{\int_H \prod_{i=1}^n \frac{p_{\theta_n(h_n),\eta}}{p_0}(X_i)
      \,d\Pi_H(\eta) < e^{-(1+C)n\rho^2}\Pi_H(K_n(\rho,M))\bigg\}
   \cap \{h_n \leq \Delta_n\} \bigg) \leq \frac{1}{C^2n\rho^2},
\]
for all $C>0$, $\rho > 0$ and $n \geq 1$, where
$\theta_n(h_n) = \theta_0 + n^{-1}h_n$.
\end{lemma}
The proof of this lemma can be found in Section~\ref{sec:proofs}.

In many applications, $(\rho_n)$ does not play an explicit role because
consistency \emph{at some rate} is sufficient. The following provides
a possible formulation of weakened conditions guaranteeing consistency
under perturbation. Corollary~\ref{cor:PertRateFree} is
based on the family of Kullback--Leibler neighbourhoods that would also play
a role for marginal posterior consistency of the nuisance with
\emph{known} $\tht_0$ (as in \cite{Ghosal00a}):
\[
  K(\rho)=\Bigl\{\eta \in H\,:\, -P_0\log\frac{p_{\theta_0,\eta}}{p_0}
    \leq\rho^2,
    P_0\Bigl(\log\frac{p_{\theta_0, \eta}}{p_0}\Bigr)^2\leq\rho^2\Bigr\},
\]
for $\rho>0$.
\begin{corollary}
\label{cor:PertRateFree}
Assume that for all $\rho>0$, $N(\rho,H,d_H)<\infty$ and
$\Pi_H(K(\rho))>0$. Furthermore, assume that for every stochastic,
bounded $(h_n)$,
\begin{itemize}
\item[(i)] for every $M>0$, there exists an $L>0$ such that for all
  $\rho>0$ and large enough $n$, $K(\rho)\subset K_n(L\rho, M)$. 
\item[(ii)] Hellinger distances satisfy
  $\sup_{\eta \in H} H(P_{\theta_n(h_n),\eta}, P_{\theta_0, \eta}) = O(n^{-1/2})$.
\end{itemize}
Then there exists a sequence $(\rho_n)$, $\rho_n\downarrow0$,
$n\rho_n^2\to\infty$, such that the conditional nuisance posterior
converges under $n^{-1}$-perturbation at rate $(\rho_n)$.
\end{corollary}
\begin{proof}
See the proof of Corollary~3.3 in Bickel and Kleijn (2012) \cite{Bickel12}.
\end{proof}

\subsection{Marginal posterior asymptotic exponentiality}
\label{sub:MPAE}

To see how the irregular Bernstein--von~Mises assertion (\ref{eq:assertion})
arises, we note the following: the marginal posterior density
$\pi_n:\Theta\to\RR$ for the parameter of interest with respect to the
prior $\Pi_\Theta$ is given by,
\[
  \pi_n(\theta)
  =\int_H\prod_{i=1}^n\frac{p_{\theta,\eta}}{p_{\theta_0,\eta_0}}(X_i)\,d\Pi_H(\eta)
    \Bigg /
  \int_\Theta\int_H\prod_{i=1}^n\frac{p_{\theta,\eta}}{p_{\theta_0,\eta_0}}(X_i)\,
    d\Pi_H(\eta)\,d\Pi_\Theta(\theta),
\]
$P^n_0$-almost-surely. This form resembles that of a \emph{parametric}
posterior density on $\Theta$ if one replaces the ordinary,
parametric likelihood by the integral of the semiparametric
likelihood with respect to the nuisance prior, \cf\ $S_n(\tht)$ in
(\ref{eq:IL}). If $S_n(\tht)$ displays properties similar to those that
lead to posterior asymptotic normality in the smooth parametric case,
we may hope that in the irregular, semiparametric setting the classical
proof can be largely maintained. More specifically, we shall replace
the LAN expansion of the parametric likelihood by a stochastic LAE
expansion of the likelihood integrated over the nuisance as in
(\ref{eq:IL}). Theorem~\ref{thm:PostLAE} uses this observation
to reduce the proof of the main theorem of this paper to a strictly
parametric discussion.

In this subsection, we prove marginal posterior asymptotic exponentiality
in two parts: first we show that $S_n(\tht)$ satisfies an LAE expansion
of its own, and second, we use this to obtain Bernstein--von~Mises
assertion (\ref{eq:assertion}), proceeding along the lines of proofs
presented in Le~Cam and Yang (1990) \cite{LeCam90a}, Kleijn and
van~der~Vaart (2012) \cite{Kleijn12} and Kleijn (2003) \cite{Kleijn03}.
We restrict attention to the case in which the model itself is
stochastically LAE and the posterior is consistent under
$n^{-1}$-perturbation (although other, less stringent formulations are
conceivable).
\begin{theorem}[Integrated Local Asymptotic Exponentiality]
\label{thm:S-ILAE}
Suppose that the model is stochastically locally asymptotically
exponential in the $\tht$-direction at all points $(\theta_0,\eta)$,
$(\eta \in H)$ and that conditions (iii) and (iv) of Theorem~\ref{thm:LAESBvM}
are satisfied. Furthermore, assume that model and prior $\Pi_H$ are
such that for some rate $(\rho_n)$ and every bounded, stochastic $(h_n)$,
\[
  \Pi_n\bigl(\,D^c(\rho_n)\bigm|
    \theta=\theta_0+n^{-1}h_n;X_1,\ldots,X_n\,\bigr)\convprob{P_0} 0.
\]
Then the integral LAE-expansion holds, \ie,
\[
  \int_H \prod_{i=1}^n \frac{p_{\theta_n(h_n),\eta}}{p_0}(X_i)\,d\Pi_H(\eta)
    = \int_H \prod_{i=1}^n \frac{p_{\theta_0,\eta}}{p_0}(X_i)\,d\Pi_H(\eta)
      \,\exp(h_n\gamma_{\tht_0,\eta_0}+o_{P_0}(1))\1_{\{h_n \leq \Delta_n\}},
\]
for any stochastic sequence $(h_n) \subset \RR$ that is bounded in
$P_0$-probability.
\end{theorem}
The following theorem uses the above integrated LAE expansion in
conjunction with a marginal posterior convergence condition to derive
the exponential Bernstein--von~Mises assertion. Marginal posterior
convergence forms the subject of the next subsection.
\begin{theorem}[Posterior asymptotic exponentiality]
\label{thm:PostLAE}
Let $\Theta$ be open in $\RR$ with thick prior $\Pi_\Theta$.
Suppose that for every $n \geq 1$, $h \mapsto s_n(h)$ is continuous
on $(-\infty, \Delta_n]$, $P_0$-almost-surely. Assume that for every
stochastic sequence $(h_n)\subset \RR$ that is bounded in probability,
\begin{equation}
  \label{eq:ilae}
  \frac{s_n(h_n)}{s_n(0)}=\exp(h_n\gamma_{\tht_0,\eta_0}
    +o_{P_0}(1))\1_{\{h_n \leq \Delta_n\}},
\end{equation}
for some positive constant $\gamma_{\tht_0,\eta_0}$. Suppose that for every
$M_n\to\infty$, we have,
\begin{equation}
\label{eq:cons}
  \Pi_n\bigl(\,|h|\leq M_n\,\bigm|\,X_1,\ldots,X_n\,\bigr)\convprob{P_0}1.
\end{equation}
Then the sequence of marginal posteriors for $\theta$ is asymptotically
exponential in $P_0$-probability, converging in total variation to
a negative exponential distribution,
\begin{equation}\label{eq:iConv}
  \sup_A\Bigl|\,\Pi_n\bigl(\,h\in A\,\bigm|\,X_1,\ldots,X_n\,\bigr) 
    - \NExp_{\Delta_n,\gamma_{\tht_0,\eta_0}}(A) \,\Bigr|\convprob{P_0} 0.
\end{equation}
\end{theorem}
Conditions \emph{(iii)} and \emph{(iv)} of Theorem~\ref{thm:LAESBvM} are
crucial in the derivation of the two theorems presented above. In the
following lemma we present
two sufficient conditions for both the domination and the one-sided 
contiguity condition to hold. The first method poses the domination
condition in slightly stronger form (see ``q-domination'' below); the second
relies on a log-Lipschitz condition for model densities and uniform
finiteness of exponential moments of the Lipschitz constant.
\begin{lemma}
\label{lem:DomLemma}
Suppose that the model satisfies at least one of the following two conditions:
\begin{itemize}
\item[(i)] (``$q$-domination'' condition)\\
  for every bounded, stochastic $(h_n)$, small enough $\rho > 0$,
  and some $q > 1$,
  \begin{equation}
    \label{eq:domination-q}
    \sup_{\eta \in D(\rho)} P^n_{\theta_0, \eta}
      \bigg(\prod_{i=1}^n \frac{p_{\theta_n(h_n),\eta}}{p_{\theta_0,\eta}}(X_i)
      \bigg)^q = O(1),
  \end{equation}
\item[(ii)] (log-Lipschitz condition)\\
  or, for all $\eta \in H$ there exists a measurable $m_{\theta_0,\eta}>0$ 
  such that for every $x \in A_{\theta_0,\eta}$ and for every $\theta$ in a 
  neighbourhood of $\theta_0$,
  \begin{equation}
    \label{eq:LogLipschitz}
    \frac{p_{\theta,\eta}}{p_{\theta_0,\eta}}(x)
    \leq e^{m_{\theta_0,\eta}(x)|\theta_0-\theta|},
  \end{equation}
  and for small enough $\rho > 0$ and all $K > 0$,
  $\sup_{\eta\in D(\rho)} P_{\theta_0,\eta}e^{K m_{\theta_0,\eta}} < \infty$.
\end{itemize}
Then, for fixed $\rho>0$ small enough,
\begin{itemize}
  \item[(i)] the model satisfies the domination condition
    \begin{equation*}
    \label{eq:dominationA}
      \sup_{\eta \in D(\rho)} P^n_{\theta_0, \eta}
      \bigg(\prod_{i=1}^n \frac{p_{\theta_n(h_n),\eta}}{p_{\theta_0,\eta}}(X_i)
      \bigg)= O(1),
    \end{equation*}
\item[(ii)] and, for every $\eta \in D(\rho)$, the $(P^n_{\theta_n(h_n),\eta})$
  is contiguous with respect to the $(P^n_{\theta_0,\eta})$.
\end{itemize}
\end{lemma}
The log-Lipschitz version of this lemma is used in both examples
of Section~\ref{sec:Examples} to satisfy conditions \emph{(iii)}
and \emph{(iv)} of Theorem~\ref{thm:LAESBvM}.

\subsection{Marginal posterior convergence at $n^{-1}$-rate}
\label{sub:Marginal}

One of the conditions in the main theorem is marginal consistency at
rate $n^{-1}$, so that the posterior measure of a sequence of model subsets
of the form
\[
  \Theta_n \times H
    = \{(\theta, \eta) \in \Theta \times H : n|\theta-\theta_0|\leq M_n\},
\]
converge to one in $P_0$-probability, for every sequence $(M_n)$ such
that $M_n \to \infty$.
Marginal (semiparametric) posteriors have
not been studied extensively or systematically in the literature. As a
result fundamental questions (\eg\ semiparametric bias) concerning
\emph{marginal} posterior consistency have not yet received the attention
they deserve. Here, we present a straightforward formulation of sufficient
conditions, based solely on bounded likelihood ratios.
This has the advantage of leaving the nuisance prior completely unrestricted
but may prove to be too stringent a condition on the model in some
applications. Conceivably \cite{Castillo12b}, the nuisance prior has a much
more significant role to play in questions on marginal consistency. The
inadequacy of Lemma~\ref{lem:MarginalLR} manifests itself primarily
through the occurrence of a supremum over the nuisance space $H$ in 
condition~(\ref{eq:Condition}), a uniformity that is too coarse. It can be 
refined somewhat by requiring uniform bound on the likelihood ratios 
on a sequence of model subsets, capturing the most of the full 
nonparametric posterior mass. Reservations aside, it appears
from the examples of Section~\ref{sec:Examples} that the lemma is also
useful in the form stated.
\begin{lemma}
\label{lem:MarginalLR}
Let the sequence of maps $\theta \mapsto S_n(\theta)$ be $P_0$-almost 
surely continuous on $(-\infty,\Delta_n]$ and exhibit the stochastic 
integral LAE property. Furthermore, assume that there exists a constant 
$C > 0$ such that for any $(M_n)$, $M_n \to \infty$, $M_n \leq n$ for $n \geq 1$, 
and $M_n = o(n)$,
\begin{equation}
\label{eq:Condition}
	P_0^n\biggl(\sup_{\eta\in H}\sup_{\tht\in\Tht_n^c}
     \PP_n\log\frac{p_{\tht,\eta}}{p_{\tht_0,\eta}} 
     \leq - \frac{CM_n}{n}\biggr)\to 1.
\end{equation}
Then, for any nuisance prior $\Pi_H$ and $\Pi_\Theta$ that is thick at
$\theta_0$,
\[
  \Pi_n\bigl(\,n|\theta-\theta_0|>M_n \bigm|X_1,\ldots,X_n\,\bigr)
    \convprob{P_0} 0,
\]
for any $(M_n)$, $M_n \to \infty$.
\end{lemma}
\begin{proof}
Let us first note, that if marginal consistency holds for a sequence
$M_n$, then it also holds for any sequence $M'_n$ that diverges faster
(\ie\ if $M_n=O(M'_n)$). Without loss of generality, we therefore assume
that $M_n$ diverges more slowly than $n$, \ie\ $M_n = o(n)$. We can also 
assume $M_n \leq n$ for $n \geq 1$. Define $F_n$ 
to be the events in (\ref{eq:Condition}) so that $P_0^n(F_n^c) = o(1)$ 
by assumption. In addition, let
\[
  G_n = \biggl\{(X_1, \ldots, X_n) : 
    \int_\Theta S_n(\tht)\, d\Pi_\Tht(\tht) \geq e^{-CM_n/2}S_n(\tht_0)
  \biggr\}.
\]
By Lemma~\ref{lem:MarginalDenominator}, $P^n_0(G^c_n) = o(1)$ as well. Hence,
\begin{align*}
    P^n_0&\Pi_n\bigl(n|\theta-\theta_0|>M_n\bigm|X_1,\ldots,X_n\bigr)\\[2mm]
      &\leq P^n_0\Pi_n\bigl(n|\theta-\theta_0|>M_n\bigm|\sample_n\bigr)
        \1_{F_n\cap G_n}(\sample_n)+o(1)\\
      &\leq e^{CM_n/2}P_0^n\biggl(\frac{1}{S_n(\tht_0)}\!\int_H\!\!\int_{\Tht_n^c}\!
         \prod_{i=1}^n\frac{p_{\tht,\eta}}{p_{\tht_0,\eta}}(X_i) 
         \prod_{i=1}^n\frac{p_{\tht_0,\eta}}{p_{\tht_0,\eta_0}}(X_i)
         \,d\Pi_\Tht\, d\Pi_H\, 1_{F_n}(\sample_n)\biggr)\\
      &\qquad\qquad{+} o(1). 
%
\end{align*}
On the events $F_n$ we have
\begin{align*}
\int_H\!\!\int_{\Tht_n^c}\!
         &\prod_{i=1}^n\frac{p_{\tht,\eta}}{p_{\tht_0,\eta}}(X_i)
         \prod_{i=1}^n\frac{p_{\tht_0,\eta}}{p_{\tht_0,\eta_0}}(X_i)
         \,d\Pi_\Tht\, d\Pi_H\\
         & = \int_H\prod_{i=1}^n\frac{p_{\tht_0,\eta}}{p_{\tht_0,\eta_0}}(X_i)\int_{\Tht_n^c}\!
         \exp\Bigl(n\PP_n\log\frac{p_{\tht,\eta}}{p_{\tht_0,\eta}}\Bigr)
         \,d\Pi_\Tht\, d\Pi_H\\
         &\leq \int_H\prod_{i=1}^n\frac{p_{\tht_0,\eta}}{p_{\tht_0,\eta_0}}(X_i)\, d\Pi_H 
         \sup_{\eta\in H}\sup_{\tht\in\Tht_n^C}
         \exp\Bigl(n\PP_n\log\frac{p_{\tht,\eta}}{p_{\tht_0,\eta}}\Bigr)\\
         &\leq S_n(\tht_0)\exp\Bigl(\sup_{\eta\in H}\sup_{\tht\in\Tht_n^C}
         n\PP_n\log\frac{p_{\tht,\eta}}{p_{\tht_0,\eta}}\Bigr),
\end{align*}
which ultimately proves marginal consistency at rate $n^{-1}$.
\end{proof}
In the proof of Lemma~\ref{lem:MarginalLR} the lower bound for the
denominator of the marginal posterior comes from the following lemma.
(Let $\Pi_n$ denote the prior $\Pi_{\Tht}$ in the local parametrization
in terms of $h=n(\tht-\tht_0)$.)

\begin{lemma}
\label{lem:MarginalDenominator}
Let the sequence of maps $\theta \mapsto s_n(\theta)$ exhibit the LAE
property of (\ref{eq:ilae}). Assume that the prior $\Pi_\Theta$ is
thick at $\theta_0$ (and denoted by $\Pi_n$ in the local parametrization
in terms of $h$). Then
\[
  P^n_0\Bigl(\int s_n(h)\, d\Pi_n(h) < a_n s_n(0)\Bigr) \to 0,
\]
for every sequence $(a_n)$, $a_n \downarrow 0$.
\end{lemma}


\section{Proofs}
\label{sec:proofs}

In this section, several longer proofs of theorems and lemmas in the main
text have been collected.

\subsection{Proof of Lemma~\ref{lem:GGVlike}}


\begin{proof} (of Lemma~\ref{lem:GGVlike})\\
Let $C>0$, $\rho>0$, and $n\geq1$ be given. If $\Pi_H(K_n(\rho,M))=0$,
the assertion holds trivially, so we assume $\Pi_H(K_n(\rho,M))>0$ without
loss of generality and consider the conditional prior
$\Pi_n(A) =\Pi_H(A|K_n(\rho, M))$ (for measurable $A\subset H$). Since,
\[
  \int_H\prod_{i=1}^n\frac{p_{\theta_n(h_n),\eta}}{p_0}(X_i)\,d\Pi_H(\eta)
    \geq \Pi_H(K_n(\rho,M))
    \int\prod_{i=1}^n\frac{p_{\theta_n(h_n),\eta}}{p_0}(X_i) \, d\Pi_n(\eta),
\]
we may choose to consider only the neighbourhoods $K_n$. Restricting
attention to the event $\{h_n\leq\Delta_n\}$, we obtain,
\[
  \begin{split}
  \log\int&\prod_{i=1}^n\frac{p_{\theta_n(h_n),\eta}}{p_0}(X_i)\,d\Pi_n(\eta)
    \geq \int n\PP_n\log\1_{A_{\theta_n(h_n),\eta}}
      \frac{p_{\theta_n(h_n),\eta}}{p_0}\,d\Pi_n(\eta)\\
    &\geq \int\inf_{|h|\leq M}n\PP_n\1_{A_{\theta_n(h),\eta}}
      \log\frac{p_{\theta_n(h),\eta}}{p_0}\,d\Pi_n(\eta)
    \geq \int n\PP_n\inf_{|h|\leq M}\1_{A_{\theta_n(h),\eta}}
      \log \frac{p_{\theta_n(h),\eta}}{p_0}\, d\Pi_n(\eta)\\
    &\geq\sqrt{n}\int -\GG_n\biggl(\sup_{|h|\leq M}-\1_{A_{\theta_n(h),\eta}}
      \log\frac{p_{\theta_n(h),\eta}}{p_0}\biggr)\,d\Pi_n(\eta)
      -n \rho^2,\\
  \end{split}
\]
using the definition of $K_n$ in the last step (see (\ref{eq:KLngb})). Then,
\[
  \begin{split}
    P^n_0\biggl(\biggl\{\int&\prod_{i=1}^n
      \frac{p_{\theta_n(h_n),\eta}}{p_0}(X_i)\,d\Pi_n(\eta)
        <e^{-(1+C)n\rho^2}\biggr\}\cap\{h_n\leq \Delta_n\} \biggr)\\
    &\leq P^n_0\bigg(\int -\GG_n\biggl(\sup_{|h|\leq M}-\1_{A_{\theta_n(h),\eta}}
      \log\frac{p_{\theta_n(h),\eta}}{p_0}\biggr)\,d\Pi_n(\eta)
        < -\sqrt{n}C\rho^2\bigg).
  \end{split}
\]
By Chebyshev's inequality, Jensen's inequality, Fubini's theorem and the
fact that for any $P_0$-square-integrable random variables
$Z_n$, $P^n_0(\GG_nZ_n)^2 \leq P^n_0Z_n^2$,
\[
  \begin{split}
  P^n_0\biggl(\int&-\GG_n\biggl(\sup_{|h|\leq M}-\1_{A_{\theta_n(h),\eta}}
    \log\frac{p_{\theta_n(h),\eta}}{p_0}\biggr)\,d\Pi_n(\eta)
      <-\sqrt{n}C\rho^2\biggr)\\
  &\leq\frac{1}{nC^2\rho^4}\int P^n_0
    \bigg(\GG_n\sup_{|h|\leq M}-\1_{A_{\theta_n(h),\eta}}
    \log \frac{p_{\theta_n(h),\eta}}{p_0}\bigg)^2
  \,d\Pi_n(\eta) \leq \frac{1}{nC^2\rho^2},
  \end{split}
\]
where the last step follows again from definition (\ref{eq:KLngb}).
\end{proof}

\subsection{Proof of Theorems~\ref{thm:S-ILAE} and~\ref{thm:PostLAE}, 
  and Lemmas~\ref{lem:DomLemma} and~\ref{lem:MarginalDenominator}}

\begin{proof}(of Theorem~\ref{thm:S-ILAE})\\
Let $(h_n)$ be bounded in $P_0$-probability. Throughout this proof we
write $\tht_n(h_n) = \tht_0 + n^{-1}h_n$. Let $\delta,\ep>0$ be given.
There exists a constant $M>0$ such that $P^n_0(|h_n|>M)<\delta/2$ for
all $n\geq1$. By the consistency assumption, for large enough $n$, 
\[
  P^n_0\Bigl(\,\log\Pi_n\bigl(\,D(\rho_n)\,\bigm|\,
    \tht=\tht_n;X_1,\ldots,X_n\,\bigr)\geq-\ep\,\Bigr)>1-\frac{\delta}{2}.
\] 
This implies that the posterior's numerator and denominator are
related through,
\[
  P^n_0\bigg(
    \int_H\prod_{i=1}^n\frac{p_{\theta_n(h_n),\eta}}{p_0}(X_i)\,d\Pi_H(\eta)
    \leq e^\ep \1_{\{|h_n|\leq M\}}
      \int_{D(\rho_n)}\prod_{i=1}^n\frac{p_{\theta_n(h_n),\eta}}{p_0}(X_i)
      \,d\Pi_H(\eta) \bigg)
  >1-\delta,
\]
for this $M$ and all $n$ large enough. We continue with the integral
over $D(\rho_n)$ under the restriction $|h_n|\leq M$. By stochastic
local asymptotic exponentiality for every fixed $\eta$, we have,
\[
  \prod_{i=1}^n\frac{p_{\theta_n(h_n),\eta}}{p_0}(X_i) 
    = \prod_{i=1}^n\frac{p_{\theta_0,\eta}}{p_0}(X_i)\,
      \exp(h_n\gamma_{\tht_0,\eta} +R_n(h_n,\eta;\underline{X}_n)),
\]
where the rest-term $R_n(h_n,\eta;\underline{X}_n)$ converges to zero
in $P_{\theta_0,\eta}$-probability. Define for all $\ep > 0$ the events,
\[
  F_n(\eta,\ep)=\Bigl\{\sample_n:\sup_{|h|\leq M}|h\gamma_{\tht_0,\eta}
    -h\gamma_{\tht_0,\eta_0}|\leq\ep\Bigr\},
\]
and note that $F^c_n(0,\ep) = \emptyset$. With the domination condition
\emph{(iii)} of Theorem~\ref{thm:LAESBvM}, Fatou's lemma yields:
\[
  \begin{split}
    \limsup_{n\to\infty}\int_{D(\rho_n)}
      &P^n_{\theta_n(h_n),\eta}\bigl(F_n^c(\eta,\ep)\bigr)\, d\Pi_H(\eta)\\
      &\leq\int\limsup_{n\to\infty}\1_{D(\rho_n)\setminus\{0\}}
        P^n_{\theta_n(h_n),\eta}\bigl(F_n^c(\eta,\ep)\bigr)\,d\Pi_H(\eta)=0.
  \end{split}
\]
Combined with Fubini's theorem, this suffices to conclude that
\begin{equation}
  \label{eq:Fbound}
  \int_{D(\rho_n)}\prod_{i=1}^n\frac{p_{\theta_n(h_n),\eta}}{p_0}(X_i)
    \,d\Pi_H(\eta)
  =\int_{D(\rho_n)}\prod_{i=1}^n\frac{p_{\theta_n(h_n),\eta}}{p_0}(X_i)
    \1_{F_n(\eta,\ep)}(\sample_n)\,d\Pi_H(\eta)+o_{P_0}(1),
\end{equation}
and we continue with the first term on the \rhs. For every $\eta\in H$,
define the events,
\[
  G_n(\eta,\ep)=\Bigr\{\sample_n:
    \sup_{|h|\leq M}|R_n(h,\eta;\sample_n)|\leq\ep/2\Bigr\},
\] 
and note that $P^n_{\theta_0,\eta}(G^c_n(\eta,\ep))\to0$. By the
contiguity condition \emph{(iv)} of Theorem~\ref{thm:LAESBvM},
the probabilities $P^n_{\theta_n(h_n),\eta}(G^c_n(\eta,\ep))$ converge to zero
as well. Reasoning as with the events $F_n(\eta,\ep)$, we conclude that,
\[
  \begin{split}
    \int_{D(\rho_n)}&\prod_{i=1}^n \frac{p_{\theta_n(h_n),\eta}}{p_0}(X_i)
      \1_{F_n(\eta,\ep)}(\sample_n)\,d\Pi_H(\eta)\\
    &=\int_{D(\rho_n)}\prod_{i=1}^n\frac{p_{\theta_n(h_n),\eta}}{p_0}(X_i)
      \1_{G_n(\eta,\ep)\cap F_n(\eta,\ep)}(\sample_n)\,d\Pi_H(\eta)+o_{P_0}(1).
  \end{split}
\]
For fixed $n$ and $\eta$ and for all $\sample_n\in G_n(\eta,\ep)\cap
F_n(\eta,\ep)$, and by stochastic local asymptotic exponentiality,
\[
  \begin{split}
  \Biggl|\,\,\log\prod_{i=1}^n\frac{p_{\tht_n(h_n),\eta}}{p_0}(X_i)
    &- \log\prod_{i=1}^n\frac{p_{\theta_0,\eta}}{p_0}(X_i)
    -h_n\gamma_{\tht_0,\eta_0}\,\,\Biggr| \\
  &\leq \bigl|R_n(h_n,\eta;\sample_n)\bigr|
    +\bigl|h_n(\gamma_{\tht_0,\eta_0}-\gamma_{\tht_0,\eta})\bigr|\leq 2\ep,
  \end{split}
\]
from which it follows that,
\[
  \begin{split}
  \exp(h_n\gamma_{\tht_0,\eta_0} - 2\ep)
    &\int_{D(\rho_n)}\prod_{i=1}^n\frac{p_{\theta_0,\eta}}{p_0}(X_i)
       \1_{G_n(\eta,\ep)\cap F_n(\eta,\ep)}(\sample_n)\,d\Pi_H(\eta)\\
    &\leq\int_{D(\rho_n)}\prod_{i=1}^n\frac{p_{\theta_n(h_n),\eta}}{p_0}(X_i)
       \1_{G_n(\eta,\ep)\cap F_n(\eta,\ep)}(\sample_n)\,d\Pi_H(\eta)\\
    &\leq\exp(h_n\gamma_{\tht_0,\eta_0} + 2\ep)
     \int_{D(\rho_n)}\prod_{i=1}^n\frac{p_{\theta_0,\eta}}{p_0}(X_i)
       \1_{G_n(\eta,\ep)\cap F_n(\eta,\ep)}(\sample_n)\,d\Pi_H(\eta).
  \end{split}
\]
The integrals can be relieved of indicators for $G_n\cap F_n$ by
reversing preceding arguments (with $\tht_0$ replacing $\tht_n$),
at the expense of an $\exp(o_{P_0}(1))$-factor, leading to,
\[
  \begin{split}
    \exp(h_n\gamma_{\tht_0,\eta_0} - 3\ep+o_{P_0}(1))
    &\int_{H}\prod_{i=1}^n\frac{p_{\theta_0,\eta}}{p_0}(X_i)\,d\Pi_H(\eta)\\
    &\leq \int_{H}\prod_{i=1}^n\frac{p_{\theta_n(h_n),\eta}}{p_0}(X_i)
       \,d\Pi_H(\eta)\\
    &\qquad\leq\exp(h_n\gamma_{\tht_0,\eta_0} + 3\ep+o_{P_0}(1))
       \int_{H}\prod_{i=1}^n\frac{p_{\theta_0,\eta}}{p_0}(X_i)\,d\Pi_H(\eta).
  \end{split}
\]
for all $h_n \leq \Delta_n$. Since this holds for arbitrarily small $\ep>0$,
it proves desired result.
\end{proof}

\begin{proof}(of Theorem~\ref{thm:PostLAE})\\
Let $C$ be an arbitrary compact subset of $\RR$ containing an open
neighbourhood of the origin. Denote the (randomly located) distribution
$\NExp_{\Delta_n,\gamma_{\tht_0,\eta_0}}$ by $\Xi_n$. The prior and marginal
posterior for the local parameter $h$ are denoted $\Pi_n$ and
$\Pi_n(\,\cdot\,|\sample_n)$. Conditioned on $C\subset\RR$, these
measures are denoted $\Xi_n^C, \Pi_n^C$ and $\Pi^C_n(\,\cdot\,|\sample_n)$
respectively. Define the functions $\xi^*_n,\xi_n:\RR\to\RR$ as,
\[
  \xi^*_n(x) = \gamma_{\tht_0,\eta_0} e^{\gamma_{\tht_0,\eta_0}(x-\Delta_n)},\quad
  \xi_n(x) = \xi^*_n(x)\,\1_{\{x\leq \Delta_n\}}.
\]
noting that $\xi_n$ is the Lebesgue density for $\Xi_n$. Also
define $s^*_n(h)=s_n(h)$ on $(-\infty,\Delta_n]$ and
$s^*_n(h)=s_n(0)\exp(h\gamma_{\tht_0,\eta_0}+d_n)$ elsewhere. Finally, define,
for every $g,h\in C$ and large enough $n$,
\[
  f_n(g,h)=\left(1-\frac{\xi_n(h)}{\xi_n(g)}
    \frac{s_n(g)}{s_n(h)} \frac{\pi_n(g)}{\pi_n(h)}\right)_+
    \1_{\{g\leq \Delta_n\}}
    \1_{\{h\leq \Delta_n\}},
\]
and
\[
  f^*_n(g,h)=\left(1-\frac{\xi^*_n(h)}{\xi^*_n(g)}
    \frac{s^*_n(g)}{s^*_n(h)} \frac{\pi_n(g)}{\pi_n(h)}\right)_+,
\]
By (\ref{eq:ilae}) we know that $d_n=\log s_n(\Delta_n)-\log s_n(0)
-\Delta_n\gamma_{\tht_0,\eta_0} = o_{P_0}(1)$. Furthermore, for every stochastic
sequence $(h_n)$ in $C$,
\[
  \log s^*_n(h_n) = \log s^*_n(0) + h_n\gamma_{\tht_0,\eta_0}+o_{P_0}(1),\quad
  \log \xi^*_n(h_n)  = (h_n-\Delta_n)\gamma_{\tht_0,\eta_0}
    + \log \gamma_{\tht_0,\eta_0}.
\]
Since $\xi_n^*(h)$ and $\xi_n(h)$ ($s^*_n(h)$ and $s_n(h)$, respectively)
coincide on $\{h \leq \Delta_n\}$, $f_n(g,h) \leq f^*_n(g,h)$. For any
two stochastic sequences $(h_n),(g_n)$ in $C$, $\pi_n(g_n)/\pi_n(h_n)\to1$
as $n\to\infty$ since $\pi$ is continuous and non-zero at $\theta_0$.
Combination with the above display leads to,
\[
  \log \frac{\xi^*_n(h)}{\xi^*_n(g)}
    \frac{s^*_n(g)}{s^*_n(h)}\frac{\pi_n(g)}{\pi_n(h)}
  = (h_n-\Delta_n)\gamma_{\tht_0,\eta_0} - (g_n-\Delta_n)\gamma_{\tht_0,\eta_0}
    +g_n\gamma_{\tht_0,\eta_0} - h_n\gamma_{\tht_0,\eta_0} + o_{P_0}(1)
  =  o_{P_0}(1).
\]
Since $x \mapsto (1-e^x)_+$ is continuous on $(-\infty, \infty)$, we
conclude that for any stochastic sequence $(g_n,h_n)$ in $C\times C$,
$f^*_n(g_n,h_n) \convprob{P_0} 0$. To render this limit uniform over
$C\times C$, continuity is enough: $(g,h)\mapsto\pi_n(g)/\pi_n(h)$ is
continuous since the prior is thick. Note that $\xi^*_n(h)/s^*_n(h)$ is
of the form $\gamma_{\tht_0,\eta_0}\exp(\gamma_{\tht_0,\eta_0}
(\Delta_n+R_n(h)))$ for all $h$,
$n\geq1$, and $R_n(h_n)=o_{P_0}(1)$. Tightness of $\Delta_n$ and $R_n$
implies that $\xi^*_n(h)/s^*_n(h) \in (0,\infty)$, $(P^n_0-a.s.)$.
Continuity of $h\mapsto s_n(h)$ and $h\mapsto \xi^*_n(h)$ then implies
continuity of $(g,h) \mapsto (\xi^*_n(h)s_n^*(g))/(\xi^*_n(g)s_n^*(h))$,
$(P^n_0-a.s.)$. Hence we conclude that,
\begin{equation}
\label{eq:uni}
  \sup_{(g,h)\in C\times C} f_n(g,h)\leq\sup_{(g,h)\in C\times C}f^*_n(g,h)
  \convprob{P_0} 0.
\end{equation}
Since $s_n(h)$ is supported on $(-\infty,\Delta_n]$, since $C$ contains
a neighbourhood of the origin and since $\Delta_n$ is tight and positive,
$\Xi_n(C)>0$ and $\Pi_n(C|\sample_n)>0$, $(P^n_0-a.s.)$.
So conditioning on $C$ is well-defined (for the relevant cases where
$h \leq \Delta_n$). Let $\delta >0$ be given and define events,
\[
  \Omega_n=\Bigl\{\sample_n:\sup_{(g,h)\in C\times C}f_n(g,h)\leq\delta\Bigr\}.
\]
Based on $\Omega_n$ and (\ref{eq:uni}), write,
\[
  P^n_0\sup_A\Bigl|\,\Pi^C_n(h\in A|\underline{X}_n)-\Xi_n^C(A)\,\Bigr|
  \leq P^n_0\sup_A\Bigl|\,\Pi^C_n(h\in A|\sample_n)-\Xi_n^C(A)\,\Bigr|\,
    \1_{\Omega_n} + o(1).
\]
Note that both $\Xi_n^C$ and $\Pi_n^C(\cdot|\sample_n)$ have strictly
positive densities on $C$. Therefore, $\Xi_n^C$ is dominated by
$\Pi_n^C(\cdot|\sample_n)$ for all $n$ large enough. With that observation,
the first term on the right-hand side of the above display is calculated
to be,
\[
  \begin{split}
  \frac{1}{2}&P^n_0\sup_A\Bigl|\,\Pi^C_n(h\in A|\sample_n)-\Xi_n^C(A)\,\Bigr|
    \1_{\Omega_n}(\sample_n)\\
  &= P^n_0\int_C\Bigl(1-\frac{d\Xi^C_n}{d\Pi^C_n(\cdot|\sample_n)}\Bigr)_+
    \1_{\{h\leq\Delta_n\}}\,d\Pi^C_n(h|\sample_n)\1_{\Omega_n}(\sample_n)\\
  &= P^n_0\int_C\Bigl(1-\xi_n^C(h)\frac{\int_Cs_n(g)\pi_n(g)
    \1_{\{g\leq\Delta_n\}}dg}{s_n(h)\pi_n(h)}\Bigr)_+
    \1_{\{h\leq\Delta_n\}}d\Pi^C_n(h|\sample_n)\1_{\Omega_n}(\sample_n)\\
  &= P^n_0\int_C\Bigl(1-\int_C\frac{s_n(g)\pi_n(g)\xi_n(h)}
    {s_n(h)\pi_n(h)\xi_n(g)} \1_{\{g\leq\Delta_n\}}d\Xi_n^C(g)\Bigr)_+
    \!\1_{\{h\leq\Delta_n\}}d\Pi^C_n(h|\sample_n)\1_{\Omega_n}(\sample_n),\\
  \end{split}
\]
for large enough $n$. Jensen's inequality leads to
\[
  \begin{split}
  \frac{1}{2}&P^n_0\sup_A\Bigl|\,\Pi^C_n(h\in A|\sample_n)-\Xi_n^C(A)\,\Bigr|
    \1_{\Omega_n}(\sample_n)\\
  &\leq P^n_0\int\Bigl(1-\frac{s_n(g)\pi_n(g)\xi_n(h)}
    {s_n(h)\pi_n(h)\xi_n(g)}\Bigr)_+ \1_{\{h\leq\Delta_n\}}
     \,\1_{\{g\leq\Delta_n\}}
    \,d\Xi_n^C(g)\,d\Pi^C_n(h|\sample_n)\1_{\Omega_n}(\sample_n)\\
  &\leq P^n_0 \int\sup_{(g,h)\in C\times C}f_n(g,h)
    \,d\Xi_n^C(g)\,d\Pi^C_n(h|\sample_n)\1_{\Omega_n}(\sample_n)\leq\delta.
  \end{split}
\]
We conclude that for all compact $C\subset\RR$ containing a neighbourhood
of the origin, $P^n_0\|\Pi_n^C-\Xi_n^C\|\to 0$. To finish the argument, let
$(C_m)$ be a sequence of closed balls centred at the origin with radii
$M_m\to\infty$. For each fixed $m\geq1$ the above display holds with
$C = C_m$, so if we traverses the sequence $(C_m)$ slowly enough,
convergence to zero can still be guaranteed, \ie\ there exist $(M_n)$,
$M_n\to\infty$ such that, $P^n_0\|\Pi_n^{B_n}-\Xi_n^{B_n}\|\to 0$.
Using Lemmas~2.11 and~2.12 in \cite{Kleijn03} we conclude that
(\ref{eq:iConv}) holds.
\end{proof}

\begin{proof}(of Lemma~\ref{lem:DomLemma})\\
Assume first that the ``$q$-domination'' condition is satisfied. Assertion
\emph{(i)} follows from Jensen's inequality. For the second assertion,
fix $\eta \in D(\rho)$ and take a sequence of events $(F_n)$ such that
$P^n_{\theta_0,\eta}(F_n) \to 0$. Contiguity now follows from H\"older's
inequality (with $1/p+1/q=1$),
\[
    P^n_{\theta_n(h_n),\eta}(F_n)
    \leq \Bigl(\int\Bigl(
      \prod_{i=1}^n\frac{p_{\theta_n(h_n),\eta}}{p_{\theta_0,\eta}}(X_i)
      \Bigr)^q\,dP^n_{\theta_0,\eta}\Bigr)^{1/q}
      \Bigl(\int \1_{F_n}^p\,dP^n_{\theta_0,\eta}\Bigr)^{1/p}
    \lesssim P^n_{\theta_0,\eta}(F_n)^{1/p}\to 0.
\]
Next, assume that the log-Lipschitz condition is satisfied. Let $(h_n)$ be a
stochastic sequence bounded by $M>0$. By (\ref{eq:LogLipschitz}),
\[
  \prod_{i=1}^n \frac{p_{\theta_n(h_n), \eta}}{p_{\theta_0,\eta}}(X_i)
  \leq \exp\Bigl(\sum_{i=1}^n m_{\theta_0,\eta}(X_i) \frac{|h_n|}{n}\Bigr)
  \leq \exp\Bigl(\frac{M}{n}\sum_{i=1}^n m_{\theta_0,\eta}(X_i) \Bigr),
\]
for $X_i$ in $A_{\theta_0,\eta}$, which holds with $P_{\theta_0,\eta}$-probability
one. Therefore,
\[
  P^n_{\theta_0,\eta}\bigg(
    \prod_{i=1}^n\frac{p_{\theta_n(h_n),\eta}}{p_{\theta_0,\eta}}(X_i)
    \bigg)
  \leq P^n_{\theta_0,\eta}\bigg(
    \exp\Bigl(\frac{M}{n}\sum_{i=1}^nm_{\theta_0,\eta}(X_i)\Bigr)
    \bigg)
  \leq P_{\theta_0,\eta} \exp(M m_{\theta_0,\eta}).
\]
Due to the uniformity of the assumed bound on
$P_{\theta_0,\eta}\exp(Km_{\theta_0,\eta})$, this proves $(i)$. For the
second assertion, fix $\eta\in D(\rho)$ for some $\rho>0$ small enough,
and take a sequence of events $F_n$ such that $P^n_{\theta_0,\eta}(F_n)\to0$.\
Then,
\[
  \begin{split}
  P^n_{\theta_n(h_n),\eta}(F_n)
    & \leq \int \exp\Bigl(\frac{M}{n}\sum_{i=1}^n m_{\theta_0,\eta}(X_i) \Bigr)
      \1_{F_n}(\sample_n) \, dP^n_{\theta_0,\eta}\\
    &\leq \Bigl(\int \exp\Bigl(\frac{qM}{n}\sum_{i=1}^n m_{\theta_0,\eta}(X_i)
      \Bigr) \, dP^n_{\theta_0,\eta} \Bigr)^{1/q}
      \Bigl(\int\1_{F_n}^p\,dP^n_{\theta_0,\eta} \Bigr)^{1/p}\\
    &\leq\bigl(P_{\theta_0,\eta}\exp(qMm_{\theta_0,\eta})\bigr)^{1/q}
      P^n_{\theta_0,\eta}(F_n)^{1/p}\rightarrow 0,
  \end{split}
\]
where we have used H\"older's inequality (with $1/p +1/q = 1$) and Jensen's
inequality. The uniform bound on $P_{\theta_0,\eta}\exp(K m_{\theta_0,\eta})$
implies that $\bigl(P_{\theta_0,\eta} \exp(qMm_{\theta_0,\eta})\bigr)^{1/q}$ is
finite for any $\eta\in D(\rho)$ and $q>1$.  
\end{proof}

\begin{proof}(of Lemma~\ref{lem:MarginalDenominator})\\
Let $M > 0$ be given and define the set $C=\{h:-M\leq h\leq0\}$.
Denote the $o_{P_0}(1)$ rest-term in the integral LAE expansion
(\ref{eq:ilae}) by $h\mapsto R_n(h)$. By continuity of $\theta\mapsto
S_n(\theta)$, the expansion holds uniformly over compacta for large
enough $n$ and in particular, $\sup_{h \in C}|R_n(h)|$ converges to zero
in $P_0$-probability. Let $(K_n)$, $K_n\to\infty$ be given. The events
$B_n = \bigl\{\sup_C |R_n(h)|\leq K_n/2\bigr\}$ satisfy $P^n_0(B_n)\to1$.
Since $\Pi_\Theta$ is thick at $\theta_0$, there exists
a $\pi > 0$ such that $\inf_{h\in C} d\Pi_n/dh \geq \pi$, for large
enough $n$. Therefore,
\[
  P^n_0\biggl(\int \frac{s_n(h)}{s_n(0)}\, d\Pi_n(h) \leq e^{-K_n}\biggr)
    \leq P^n_0\biggl(\biggl\{\int_C \frac{s_n(h)}{s_n(0)}\, dh
    \leq \pi^{-1}e^{-K_n}\biggl\}\cap B_n \biggr) +o(1).
\]
On $B_n$, the integral LAE expansion is lower bounded so that, for
large enough $n$,
\[
  P^n_0\biggl(\biggl\{\int_C \frac{s_n(h)}{s_n(0)}\, d\Pi_n(h)
    \leq \pi^{-1}e^{-K_n}\biggl\}\cap B_n \biggr)
    \leq P^n_0\biggl(\int_C e^{h\gamma_{\tht_0,\eta_0}}\, dh
    \leq \pi^{-1}e^{-\frac{K_n}{2}}\biggr).
\]
Since $\int_C e^{h\gamma_{\tht_0,\eta_0}}\,dh \geq
M\,e^{-M\gamma_{\tht_0,\eta_0}}$ and $K_n\to\infty$,
$e^{-\frac{K_n}{2}}\leq\pi\,M\,e^{-M\gamma_{\tht_0,\eta_0}}$ for
large enough $n$. Combination
of the above with $K_n = -\log a_n$ proves the desired result.
\end{proof}

\subsection{Proofs of Subsection~\ref{sub:semishift}}
\label{sub:proofssemishift}

We first present properties of the map defining the nuisance space.
\begin{lemma}
\label{lem:Esscher}
Let $\alpha>S$ be fixed. Define $H$ as the image of $\scrL$ under
the map that takes $\score\in\scrL$ into densities
$\eta_\score$ defined by (\ref{eq:Esscher})  
for $x\geq0$. This map is uniform-to-Hellinger continuous and
the space $H$ is a collection of probability densities that
are {\it (i)} monotone decreasing with sub-exponential tails, {\it (ii)}
continuously differentiable on $[0,\infty)$ and {\it (iii)}
log-Lipschitz with constant $\alpha+S$.
\end{lemma}
\begin{proof}
One easily shows that $\score\mapsto\exp(-\alpha\,x+\int_0^x\score)$ is
uniform-to-uniform continuous and that $\exp(-\alpha\,x+\int_0^x\score)>0$,
which implies uniform-to-Hellinger continuity of the Esscher transform.
For the properties of $\eta_\score$, note that
$\int_0^x\score(y)\,dy\leq S\,x<\alpha\,x$,
so that $x\mapsto\exp(-\alpha\,x+\int_0^x\score(t)\,dt)$ is
sub-exponential, which implies that $\score\mapsto\eta_\score$ gives rise
to a probability density. The density $\eta$ is differentiable and
monotone decreasing. Furthermore, for all $\tht,\tht_0\in\Tht$
and all $x\geq\tht_0$,
\[
 \frac{\eta_\score(x-\tht)}{\eta_\score(x-\tht_0)}
 \leq \exp\Bigl( \alpha(\tht-\tht_0) 
   + \int_{x-\tht_0}^{x-\tht}\score(t)\,dt \Bigr)
 \leq e^{(\alpha+S)|\tht-\tht_0|},
\]
proving the log-Lipschitz property.
\end{proof}

The proof of Theorem~\ref{thm:SBvM-Example} consists of a verification
of the conditions of Corollary~\ref{cor:RateFree}. The following
lemmas make the most elaborate steps explicit. 
\begin{lemma}
\label{lem:Entropy}
Hellinger covering numbers for $H$ are finite, \ie\ for all $\rho>0$, 
$N(\rho,H,d_H)<\infty$.
\end{lemma}
\begin{proof}
Given $0 < S < \alpha$, we define $\rho_0^2 = \alpha-S>0$.
Consider the distribution $Q$ with Lebesgue density $q>0$ given by
$q(x) = \rho_0^2e^{-\rho_0^2x}$ for $x \geq 0$. Then the family $\scrF=
\{x\mapsto\sqrt{{\eta_\score}/{q}(x)}:\score\in\scrL\}$
forms a subset of the collection of all monotone functions
$\RR\mapsto[0,C]$, where $C$ is fixed and depends on $\alpha$, and $S$.
Referring to Theorem~2.7.5 in van~der~Vaart and Wellner (1996)
\cite{vdVaart96}, we conclude that the $L_2(Q)$-bracketing entropy
$N_{[\, ]}(\ep,\scrF,L_2(Q))$ of $\scrF$ is finite for all $\ep>0$.
Noting that,
\[
 d_H(\eta,\eta_0)^2 = d_H\bigl(\eta_\score,\eta_{\score_0}\bigr)^2
   = \int_\RR \Bigl( \sqrt\frac{\eta_\score}{q}(x)
     - \sqrt\frac{\eta_{\score_0}}{q}(x) \Bigr)^2\,dQ(x),
\]
it follows that $N(\rho,H,d_H)=N(\rho,\scrF,L_2(Q))\leq
N_{[\, ]}(2\rho,\scrF,L_2(Q))<\infty$.
\end{proof}
The following lemma establishes that condition \emph{(ii)} of 
Corollary~\ref{cor:RateFree} is satisfied. Moreover, assuming that the nuisance 
prior is such that $\scrL\subset\supp(\Pi_\scrL)$, this lemma establishes that 
$\Pi_H(K(\rho)) > 0$. This, together with the assertion of the previous lemma, 
verifies condition \emph{(i)} of Corollary~\ref{cor:RateFree}.
\begin{lemma}
\label{lem:SupInKL}
For every $M > 0$ there exist constants $L_1, L_2>0$ such that for small 
enough $\rho>0$,
$\{ \eta_\score \in H :\|\score-\score_0\|_\infty\leq\rho^2 \} \subset
K(L_1\rho)\subset K_n(L_2\rho,M)$.
\end{lemma}
\begin{proof}
Let $\rho$, $0<\rho<\rho_0$ and $\score\in\scrL$ such that
$\|\score-\score_0\|_\infty\leq\rho^2$ be given. 
Then,
\begin{equation}
 \label{eq:diffs}
 \Bigl| \,\,\log\frac{p_{\tht_0,\eta}}{p_{\tht_0,\eta_0}}(x) 
   -\int_0^{x-\tht_0} (\score-\score_0)(t)\,dt \,\, \Bigr|
   \leq \rho^2\,P_0(X-\tht_0) + O(\rho^4),
\end{equation}
for all $x\geq\tht_0$. Define, for all $\alpha>S$ and $\score\in\scrL$, 
the logarithm $z$ of the normalising factor in (\ref{eq:Esscher}). 
Then the relevant log-density-ratio can be written as,
\[
 \log\frac{p_{\tht_0,\eta}}{p_{\tht_0,\eta_0}}(x)
   = \int_0^{x-\tht_0}(\score-\score_0)(t)\,dt 
     - z(\alpha,\score) + z(\alpha,\score_0),
\]
where only the first term is $x$-dependent. 
Assume that $\score\in\scrL$ is such that $\|\score-\score_0\|_\infty<\rho^2$.
Then, $|\int_0^{y-\tht_0}(\score-\score_0)(t)\,dt|\leq \rho^2(y-\tht_0)$,
so that $z(\alpha-\rho^2,\score_0)\leq z(\alpha,\score)\leq
z(\alpha+\rho^2,\score_0)$. Noting that
$d^kz/d{\alpha}^k(\alpha,\score_0)=(-1)^k P_0(X-\tht_0)^k<\infty$
and using the first-order Taylor expansion of $z$ in $\alpha$, we find,
$z(\alpha\pm\rho^2,\score_0) = z(\alpha,\score_0)\mp
\rho^2\,P_0(X-\tht_0) + O(\rho^4)$, and (\ref{eq:diffs}) follows.
Next note that, for every
$k\geq1$,
\begin{equation}
 \label{eq:powers}
 \Biggl|\,\,P_0\Bigl(\int_0^{X-\tht_0}(\score-\score_0)(t)\,dt\Bigr)^k\,\,\Biggr|
   \leq \rho^{2k} \int_{\tht_0}^\infty\Bigl(\int_{0}^{x-\tht_0}dy\Bigr)^k\,dP_0
   = \rho^{2k}\,P_0(X-\tht_0)^k,
\end{equation}
Using (\ref{eq:diffs}) we bound the differences between KL
divergences and integrals of scores as follows:
\[
 \begin{split}
 \Biggl|\,\, \Bigl(\log\frac{p_{\tht_0,\eta}}{p_{\tht_0,\eta_0}}(x)\Bigr) 
   &- \Bigl(\int_0^{x-\tht_0}(\score-\score_0)(t)\,dt\Bigr) \,\,\Biggr|
   \leq
   \rho^2\bigl( P_0(X-\tht_0) + O(\rho^2) \bigr),\\
 \Biggl|\,\, \Bigl(\log\frac{p_{\tht_0,\eta}}{p_{\tht_0,\eta_0}}(x)\Bigr)^2 
   &- \Bigl(\int_0^{x-\tht_0}(\score-\score_0)(t)\,dt\Bigr)^2 \,\,\Biggr|
   \leq
     \rho^2\bigl( P_0(X-\tht_0) + O(\rho^2)  \bigr)\\
       &\times\Bigl|\,\, 2\int_0^{x-\tht_0}(\score-\score_0)(t)\,dt
       + \rho^2\bigl( P_0(X-\tht_0) + O(\rho^2)  \bigr) \,\Bigr|,
 \end{split}
\]
and, combining with the bounds (\ref{eq:powers}), we see that,
\[
 \begin{split}
 -P_0\log\frac{p_{\tht_0,\eta}}{p_{\tht_0,\eta_0}}
   &\leq 2\rho^2\bigl( P_0(X-\tht_0) + O(\rho^2) \bigr),\\
 P_0\Bigl(\log\frac{p_{\tht_0,\eta}}{p_{\tht_0,\eta_0}}\Bigr)^2 
   &\leq
     \rho^4\bigl( P_0(X-\tht_0)^2 + 3P_0(X-\tht_0) + O(\rho^2)\bigr),
 \end{split}
\]
which proves the first inclusion. Let $M > 0$. Note that $A_{\tht,\eta} = 
[\tht,\infty)$ for every $\eta$, and that 
\[
  \begin{split}
    \sup_{|h|\leq M} -\1_{A_{\theta_n(h),\eta}}&
      \log\frac{p_{\theta_n(h),\eta}}{p_{\theta_0,\eta_0}}
    = \sup_{|h|\leq M} \1_{A_{\theta_n(h),\eta}}
      \log\frac{p_{\theta_0,\eta_0}}{p_{\theta_n(h),\eta}}
    = \sup_{|h|\leq M} \1_{A_{\theta_n(h),\eta}}
      \log\frac{p_{\theta_0,\eta_0}}{p_{\theta_0,\eta}}
      \frac{p_{\theta_0,\eta}}{p_{\theta_n(h),\eta}}\\
   &= \sup_{|h|\leq M} \1_{A_{\theta_n(h),\eta}}
      \log\frac{p_{\theta_0,\eta}}{p_{\theta_n(h),\eta}}
      + \log \frac{p_{\theta_0,\eta_0}}{p_{\theta_0,\eta}}
    \leq \frac{(\alpha+S) M}{n}
      + \log \frac{p_{\theta_0,\eta_0}}{p_{\theta_0,\eta}},
  \end{split}
\]
so that,
\[
  \begin{split}
  P_0\Bigl(\sup_{|h|\leq M} &-\1_{A_{\theta_n(h),\eta}}
    \log \frac{p_{\theta_n(h),\eta}}{p_{\theta_0,\eta_0}}\Bigr)
  \leq -P_0 \log \frac{p_{\theta_0,\eta}}{p_{\theta_0,\eta_0}}
    + \frac{(\alpha+S) M}{n},\\
  P_0\Bigl(\sup_{|h|\leq M} &-\1_{A_{\theta_n(h),\eta}}
    \log\frac{p_{\theta_n(h),\eta}}{p_{\theta_0,\eta_0}}\Bigr)^2\\
  &\leq  P_0 \Bigl( \log \frac{p_{\theta_0,\eta}}{p_{\theta_0,\eta_0}} \Bigr)^2
    + \frac{2(\alpha+S)M}{n}\Bigl[P_0\Bigl(
      \log \frac{p_{\tht_0,\eta}}{p_{\tht_0,\eta_0}}\Bigr)^2\Bigr]^{1/2}
    +\frac{(\alpha+S)^2 M^2}{n^2},
  \end{split}
\]
implying the existence of a constant $L_2$.
\end{proof}
By Lemma~\ref{lem:Esscher}, the log-Lipschitz constant $m_{\tht_0,\eta}$ of
Lemma~\ref{lem:DomLemma} equals $\alpha+S$ for every $\eta \in H$, so
that the domination condition \emph{(iii)} and contiguity requirement
\emph{(iv)} of Corollary~\ref{cor:RateFree} are satisfied.
The following lemma shows that condition \emph{(v)} of
Corollary~\ref{cor:RateFree} is also satisfied.
\begin{lemma}
\label{lem:Hellinger}
For all bounded, stochastic sequences $(h_n)$, Hellinger distances
between $P_{\tht_n(h_n),\eta}$ and $P_{\tht_0,\eta}$ are of order $n^{1/2}$
uniformly in $\eta$, \ie\
$\sup_{\eta\in H} n^{1/2}\,H(P_{\theta_n(h_n),\eta},P_{\theta_0,\eta}) = O(1)$.
\end{lemma}
\begin{proof}
Fix $n$ and $\omega$; write $h_n$ for $h_n(\omega)$. First we consider
the case that $h_n \geq 0$, for $x\geq\tht_0$,
\[
  \begin{split}
  \bigl(\eta^{1/2}(x-&\theta_n(h_n))-\eta^{1/2}(x-\theta_0)\bigr)^2\\
  &= \eta(x-\theta_0)\1_{[\theta_0,\theta_n(h_n)]}(x)
    + (\eta^{1/2}(x-\theta_n(h_n))-\eta^{1/2}(x-\theta_0))^2
      \1_{[\theta_n(h_n),\infty)}(x)
  \end{split}
\]
To upper bound the second term, we use the absolute continuity of $\eta^{1/2}$,
\[
  \bigl|\eta^{1/2}(x-\theta_0)-\eta^{1/2}(x-\theta_n(h_n))\bigr|
    = \frac12\Bigl|\int_{x-\theta_0-\frac{h_n}{n}}^{x-\theta_0}
        \frac{\eta'}{\eta^{1/2}}(y)\,dy\Bigr|
    \leq \frac12\int_0^{\frac{M}{n}}
        \Bigl|\frac{\eta'}{\eta^{1/2}}(z+x-\theta_n(h_n))\Bigr|\, dz,
\]
and then by Jensen's inequality,
\[
  \bigl(\eta^{1/2}(x-\theta_0)-\eta^{1/2}(x-\theta_n(h_n))\bigr)^2
  \leq \frac{M}{4n}\int_0^{\frac{M}{n}}
    \frac{(\eta')^2}{\eta}(z+x-\theta_n(h_n))\, dz.
\]
Similarly for $h_n < 0$ and $x\geq\tht_n(h_n)$,
\[
  \begin{split}
  \bigl(\eta^{1/2}&(x-\theta_0)-\eta^{1/2}(x-\theta_n(h_n))\bigr)^2\\
  &\leq\eta(x-\theta_n(h_n))\1_{[\theta_n(h_n),\theta_0]}(x)
    -\eta(x-\theta_n(-M))\1_{[\theta_n(-M),\theta_0]}(x)\\
  &\qquad +\eta(x-\theta_n(-M))\1_{[\theta_n(-M),\theta_0]}
    +\frac{M}{4n}\int_0^{\frac{M}{n}}
       \frac{(\eta')^2}{\eta}(z+x-\theta_0)\,dz\,\1_{[\theta_0,\infty)}(x).
  \end{split}
\]
Combining these results, we obtain a bound for the squared Hellinger distance:
\begin{equation}
  \label{eq:Hdistance}
  \begin{split}
    H^2(P_{\theta_n(h_n),\eta}&,P_{\theta_0,\eta})
    \leq \int_{\theta_0}^{\theta_n(M)} \eta(x-\theta_0) \, dx
      + \int_{\theta_n(-M)}^{\theta_0} \eta(x-\theta_n(-M)) \, dx \\
    &+\1_{\{h_n<0\}}\int_{\theta_n(h_n)}^{\theta_0}
        \eta(x-\theta_n(h_n)) \, dx
        -\1_{\{h_n<0\}}\int_{\theta_n(-M)}^{\theta_0} \eta(x-\theta_n(-M))\,dx\\
    &+\1_{\{h_n \geq 0\}}\int_{\theta_n(h_n)}^{\infty}\frac{M}{4n}
        \int_0^{\frac{M}{n}}\frac{(\eta')^2}{\eta}(z+x-\theta_n(h_n))\,dz\,dx\\
    &+\1_{\{h_n < 0\}}\int_{\theta_0}^{\infty}\frac{M}{4n}
        \int_0^{\frac{M}{n}}\frac{(\eta')^2}{\eta}(z+x-\theta_0)\,dz\,dx.
  \end{split}
\end{equation}
As for the first two terms on the right-hand side of (\ref{eq:Hdistance}),
we note the following inequality:
\[
  \int_{\theta_0}^{\theta_n(M)} \eta(x-\theta_0)\,dx
    +\int_{\theta_n(-M)}^{\theta_0} \eta(x-\theta_n(-M))\,dx
 \leq 2\gamma_{\tht_0,\eta}\frac{M}{n} 
  + \frac{M^2}{n^2}\int_0^\infty|\eta'(y)|\, dy,
\]
by Lemma~\ref{lem:IntBounds}. Furthermore, by shifting appropriately,
we find that the third and fourth term of (\ref{eq:Hdistance}) satisfy
the bound,
\[
  \begin{split}
  \1_{\{h_n<0\}}\Bigl(\int_{\theta_n(h_n)}^{\theta_0} &\eta(x-\theta_n(h_n))\,dx
    -\int_{\theta_n(-M)}^{\theta_0}\eta(x-\theta_n(-M))\,dx\Bigr)\\
  &=\1_{\{h_n<0\}}\Bigl(\int_{0}^{-\frac{h_n}{n}}\eta(y)\,dy
    -\int_{0}^{\frac{M}{n}}\eta(y)\,dy\Bigr)
  = -\1_{\{h_n<0\}}\int_{-\frac{h_n}{n}}^{\frac{M}{n}}\eta(y)\,dy\leq 0,
  \end{split}
\]
(where it is noted that the $h_n$ dependent integral in the above display
is well defined for any $h_n$). Finally, the fifth and sixth term
of (\ref{eq:Hdistance}) are bounded by the Fisher information for
location associated with $\eta$:
\[
  \int_0^{\infty}\int_0^{\frac{M}{n}}\frac{(\eta')^2}{\eta}(z+x)\,dz\,dx
  = \int_0^{\frac{M}{n}}\int_z^{\infty} \frac{(\eta')^2}{\eta}(x)\, dx\, dz
  \leq \frac{M}{n}\int_0^\infty \frac{(\eta')^2}{\eta}(x)\, dx;
\]
Combining, we obtain the following upper bound for the relevant
Hellinger distance,
\[
  H^2(P_{\theta_n(h_n),\eta},P_{\theta_0,\eta})
  \leq 2\gamma_{\tht_0,\eta}\frac{M}{n}
   + 2\frac{M^2}{n^2}
     \biggl(\int_0^\infty\Bigl|\frac{\eta'(x)}{\eta(x)}\Bigr|\eta(x)\,dx
   + \int_0^\infty \Bigl(\frac{\eta'(x)}{\eta(x)}\Bigr)^2\eta(x)\, dx\biggr).
\]
which proves the lemma upon noting that
$|\eta'(x)|=\eta(x)|\score(x)-\alpha|\leq\eta(x)(\alpha-S)$.
\end{proof}
To verify condition {\em (vi)} of Corollary~\ref{cor:RateFree} we now check 
condition (\ref{eq:Condition}) of Lemma~\ref{lem:MarginalLR}.
\begin{lemma}
\label{lem:LR}
Let $(M_n)$, $M_n\to\infty$, $M_n \leq n$ for $n \geq 1$, $M_n=o(n)$ be given. 
Then there exists a constant $C>0$ such that the condition of 
Lemma~\ref{lem:MarginalLR} is satisfied.
\end{lemma}
\begin{proof}
Note first that for fixed $x$ and $\eta$, the map $\theta\mapsto p_{\theta,\eta}(x)$ is 
monotone increasing. Therefore
\[
	\sup_{\tht\in\Tht_n^c}\frac{1}{n}
       \log\prod_{i=1}^n\frac{p_{\tht,\eta}}{p_{\tht_0,\eta}}(X_i)
	 \leq \frac{1}{n}\log\prod_{i=1}^n\frac{\eta(X_i-\tht^*)}{\eta(X_i-\tht_0)}
       \1_{\{X_{(1)}\geq\tht^*\}}(\sample_n),
\]
where $\tht^* = X_{(1)}$ if $X_{(1)} \geq \tht_0+M_n/n$, or 
$\tht_0-M_n/n$ otherwise. We first note that $X_{(1)} < \tht_0+M_n/n$
with probability tending to one. Indeed, shifting the distribution to 
$\tht=0$, we calculate,
\[
  P^n_{0,\eta_0} \Bigl(X_{(1)} \geq \frac{M_n}{n}\Bigr)
    = \Bigl(1- \int_0^{\frac{M_n}{n}}\eta_0(x)\, dx\Bigr)^n
    \leq \exp\Bigl(-n\int_0^{\frac{M_n}{n}}\eta_0(x)\, dx\Bigr).
\]
By Lemma~\ref{lem:IntBounds}, the right-hand side of the above display is bounded
further as follows,
\[
  \exp\Bigl(-\gamma_{\tht_0,\eta_0}M_n
    + M_n\int_0^{\frac{M_n}{n}} |\eta_0'(x)|\, dx\Bigr)
  \leq \exp\Bigl(-\frac{\gamma_{\tht_0,\eta_0}}{2} M_n\Bigr),
\]
for large enough $n$. We continue with $\tht^*=\tht_0-M_n/n$. 
By absolute continuity of $\eta$ we have
\[
	\eta(X_i-\tht^*) = \eta(X_i-\tht_0) 
       + \int_{X_i-\tht_0}^{X_i-\tht^*}\eta'(y)\, dy,
\]
and the conditions on the nuisance $\eta$ yield the following bound,
\[
	\int_{X_i-\tht_0}^{X_i-\tht^*}\eta'(y)\, dy
	 \leq (\tht_0-\tht^*) (S-\alpha)\eta(X_i-\tht_0).
\]
Therefore
\[
	\frac{1}{n}\log\prod_{i=1}^n\frac{\eta(X_i\!-\!\tht^*)}{\eta(X_i\!-\!\tht_0)}
       \1_{\{X_{(1)}\geq\tht^*\}}(\sample_n)\leq
    \frac{1}{n}\log \Bigl(1\!-\!\frac{(\alpha\!-\!S)M_n}{n}\Bigr)^n 
	 \leq -\frac{(\alpha\!-\!S)M_n}{n}.
\]
If $C < \alpha-S$, the condition of Lemma~\ref{lem:MarginalLR} is 
clearly satisfied.
\end{proof}
\whitespace

To demonstrate that priors exist such that $\scrL\subset\supp(\Pi_\scrL)$,
an explicit construction based on the distribution of Brownian sample
paths is provided in the following lemma.
\begin{lemma}
\label{lem:Prior}
Let $S > 0$ be given. Let $\{W_t:t\in[0,1]\}$ be Brownian motion on
$[0, 1]$ and let $Z$ be independent and distributed $N(0,1)$. We
define the prior $\Pi_\scrL$ on $\scrL$ as the distribution of the
process,
\[
  \score(t) = S\Psi(Z+W_{\Psi(t)}),
\]
where $\Psi:[-\infty,\infty]\rightarrow[-1,1]: x\mapsto 2\arctan (x)/\pi$.
Then $\scrL\subset\supp\bigl(\Pi_\scrL\bigr)$.
\end{lemma}
\begin{proof}
Consider $C[0,1]$ with the uniform norm and its Borel $\sigma$-algebra,
equipped with the law $\Pi$ of $t\mapsto Z+W_t$, as a probability space.
Since $\Psi$ is Lipschitz, the map $f$ that takes $C[0,1]$ into $C[0,\infty]$,
$Z+W_\cdot \mapsto Z+W_{\Psi(\cdot)}$ is continuous, norm-preserving, and
Borel-to-Borel measurable. This enables the view of $C[0,\infty]$ with its
Borel $\sigma$-algebra as a probability space, with probability
measure $\Pi'(B) = \Pi\bigl(f^{-1}(B)\bigr)$. Similarly, the map $g$ that
takes $C[0,\infty]$ into $\scrL$, $Z+W_{\Psi(\cdot)}\mapsto
S\Psi\bigl(Z+W_{\Psi(\cdot)}\bigr)$ is continuous and Borel-to-Borel
measurable. We view $\scrL$ with its Borel $\sigma$-algebra
as a probability space, with probability measure
$\Pi_\scrL(C) = \Pi'\bigl(g^{-1}(C)\bigr)$. Let $T$ denote a closed set
in $\scrL$ such that $\Pi_\scrL(T) = 1$. Note that
$f^{-1}(g^{-1}(T))$ is closed and $\Pi\bigl(f^{-1}(g^{-1}(T))\bigr) = 1$,
so that $\supp(\Pi) \subset f^{-1}(g^{-1}(T))$. Since the support
of $\Pi_\scrL$ equals the intersection of all such $T$, $\supp(\Pi)
\subset \bigcap_T f^{-1}(g^{-1}(T)) = f^{-1}\bigl(g^{-1}\bigl(
\supp\bigl(\Pi_\scrL\bigr)\bigr)\bigr)$. Since $\supp(\Pi) = C[0,1]$,
for every $y \in C[0,1]$, $f(g(y)) \in \supp\bigl(\Pi_\scrL\bigr)$.
The continuity does not change under $g\circ f$, so
$\supp\bigl(\Pi_\scrL\bigr)$ includes $\scrL$. 
\end{proof}

\subsection{Proofs of Subsection~\ref{sub:semiscale}}
\label{sub:proofssemiscale}


Again we first present properties of the mapping defining the nuisance space.
\begin{lemma}
\label{lem:Esscher2}
Define $H$ as the image of $\scrL$ under the map that takes 
$\score\in\scrL$ into densities $\eta_\score$ defined by (\ref{eq:Esscher2}) 
for $x\in[0,1]$. 
This map is uniform-to-Hellinger continuous and
the space $H$ is a collection of probability densities that
are {\it (i)} monotone increasing and bounded away from zero and 
infinity and {\it (ii)} continuously differentiable on $[0,1]$. 
Moreover, the resulting densities $p_{\tht,\eta}$ satisfy the 
log-Lipschitz condition~(\ref{eq:LogLipschitz}) in an $\ep$-neighbourhood 
$(\ep < \tht_0/2)$ with $m_{\tht_0,\eta} = (2+8S)/\tht_0$.
\end{lemma}
\begin{proof} 
The proof is similar to the proof of Lemma~\ref{lem:Esscher} 
and is therefore omitted.
\end{proof}
The proof of Theorem~\ref{thm:SBvM-Example2} consists of a verification
of the conditions of Corollary~\ref{cor:RateFree} (after the aforementioned 
modification to comply with the positive version of the LAE expansion). The 
following lemmas make the most elaborate steps explicit, as in the proof of 
Theorem~\ref{thm:SBvM-Example}.
\begin{lemma}
\label{lem:Entropy2}
Hellinger covering numbers for $H$ are finite, \ie\ for all $\rho>0$, 
$N(\rho,H,d_H)<\infty$.
\end{lemma}
\begin{proof}
Denote by $Q$ the distribution with density $\eta_0=\eta_{\score_0}$. Then 
the family $\scrF = \{x\mapsto \sqrt{\eta_\score/\eta_0} :\score\in\scrL\}$ 
forms a subset of the collection $C^1_M([0,1])$, where $M$ is fixed and 
depends on $S$. Referring to Corollary~2.7.2 in \cite{vdVaart96}, we conclude 
that the $L_2(Q)$-bracketing entropy $N_{[\, ]}(\ep,\scrF,L_2(Q))$ of $\scrF$ 
is finite for all $\ep>0$. 
Similarly as in the proof of Lemma~\ref{lem:Entropy2}, 
it follows that $N(\rho,H,d_H)=N(\rho,\scrF,L_2(Q))\leq
N_{[\, ]}(2\rho,\scrF,L_2(Q))<\infty$.
\end{proof}
The previous lemma together with the following lemma verify 
conditions \emph{(i)} and \emph{(ii)} of Corollary~\ref{cor:RateFree}.
\begin{lemma}
\label{lem:SupInKL2}
For every $M > 0$ there exist constants $L_1, L_2>0$ such that for small 
enough $\rho>0$,
$\{ \eta_\score \in H :\|\score-\score_0\|_\infty\leq\rho^2 \} \subset
K(L_1\rho) \subset K_n(L_2\rho,M)$.
\end{lemma}
\begin{proof}
Let $\rho>0$ and $\score\in\scrL$ such that
$\|\score-\score_0\|_\infty\leq\rho^2$ be given.
Then,
\begin{equation}
 \label{eq:diffs2}
 \Bigl| \,\,\log\frac{p_{\tht_0,\eta}}{p_{\tht_0,\eta_0}}(x) 
   -\int_0^{x/\tht_0} (\score-\score_0)(t)\,dt \,\, \Bigr|
   \leq \rho^2\,P_0(X/\tht_0) + O(\rho^4),
\end{equation}
for all $x\in[0,\tht_0]$. Define, for all $\alpha\in\RR$ and $\score\in\scrL$,
\[
  z(\alpha,\score) = \log\int_0^1 e^{\alpha y + \int_0^y \score(t)\, dt}\, dy.
\]
Then the relevant log-density-ratio can be written as,
\[
 \log\frac{p_{\tht_0,\eta}}{p_{\tht_0,\eta_0}}(x)
   = \int_0^{x/\tht_0}(\score-\score_0)(t)\,dt 
     - z(S,\score) + z(S,\score_0),
\]
where only the first term is $x$-dependent. 
Assume that $\score\in\scrL$ is such that $\|\score-\score_0\|_\infty<\rho^2$.
Then, $|\int_0^{y}(\score-\score_0)(t)\,dt|\leq \rho^2y$,
so that $z(S-\rho^2,\score_0)\leq z(S,\score)\leq
z(S+\rho^2,\score_0)$. Noting that
$d^kz/d{\alpha}^k(S,\score_0)=P_0(X/\tht_0)^k<\infty$
and using the first-order Taylor expansion of $z$ in $\alpha$, we find,
$z(S\pm\rho^2,\score_0) = z(S,\score_0)\pm
\rho^2\,P_0(X/\tht_0) + O(\rho^4)$, and (\ref{eq:diffs2}) follows.

Next note that, for every
$k\geq1$,
\begin{equation}
 \label{eq:powers2}
  \Biggl|\,\,
   P_0\Bigl(\int_0^{X/\tht_0}(\score-\score_0)(t)\,dt\Bigr)^k\,\,\Biggr|
   \leq \rho^{2k} \int_{0}^{\tht_0}\Bigl(\int_{0}^{x/\tht_0}dy\Bigr)^k\,dP_0
   = \rho^{2k}\,P_0(X/\tht_0)^k,
\end{equation}
Using (\ref{eq:diffs2}) we bound the differences between KL
divergences and integrals of scores
and, combining with the bounds (\ref{eq:powers2}), we see that,
\[
 \begin{split}
 -P_0\log\frac{p_{\tht_0,\eta}}{p_{\tht_0,\eta_0}}
   &\leq 2\rho^2\bigl( P_0(X/\tht_0) + O(\rho^2) \bigr),\\
 P_0\Bigl(\log\frac{p_{\tht_0,\eta}}{p_{\tht_0,\eta_0}}\Bigr)^2 
   &\leq
     \rho^4\bigl( P_0(X/\tht_0)^2 + 3P_0(X/\tht_0) + O(\rho^2)\bigr),
 \end{split}
\]
which proves the first inclusion. Let $M > 0$. Similarly as in the proof 
of Lemma~\ref{lem:SupInKL} we can show that
\[
  \begin{split}
  P_0\Bigl(\sup_{|h|\leq M} &-\1_{A_{\theta_n(h),\eta}}
    \log \frac{p_{\theta_n(h),\eta}}{p_{\theta_0,\eta_0}}\Bigr)
  \leq -P_0 \log \frac{p_{\theta_0,\eta}}{p_{\theta_0,\eta_0}}
    + \frac{2+8S}{\tht_0}\frac{M}{n},\\
  P_0\Bigl(\sup_{|h|\leq M} &-\1_{A_{\theta_n(h),\eta}}
    \log\frac{p_{\theta_n(h),\eta}}{p_{\theta_0,\eta_0}}\Bigr)^2\\
  &\leq  P_0 \Bigl( \log \frac{p_{\theta_0,\eta}}{p_{\theta_0,\eta_0}} \Bigr)^2
    + \frac{4+16S}{\tht_0}\frac{M}{n}
      \Bigl[P_0\Bigl(
      \log \frac{p_{\tht_0,\eta}}{p_{\tht_0,\eta_0}}\Bigr)^2\Bigr]^{1/2}
    +\frac{(2+8S)^2}{\tht_0^2}\frac{M^2}{n^2},
  \end{split}
\]
implying the existence of a constant $L_2$.
\end{proof}
By Lemma~\ref{lem:Esscher2} the model satisfies the Lipschitz condition
of Lemma~\ref{lem:DomLemma} with the same Lipschitz constant for every 
$\eta \in H$, so that the domination condition \emph{(iii)} and contiguity 
requirement \emph{(iv)} of Corollary~\ref{cor:RateFree} are satisfied.
\begin{lemma}
\label{lem:Hellinger2}
For all bounded, stochastic sequences $(h_n)$, Hellinger distances
between $P_{\tht_n(h_n),\eta}$ and $P_{\tht_0,\eta}$ are of order $n^{1/2}$
uniformly in $\eta$, \ie\
$\sup_{\eta\in H} n^{1/2}\,H(P_{\theta_n(h_n),\eta},P_{\theta_0,\eta}) = O(1)$.
\end{lemma}
\begin{proof}
Note that the elements of the nuisance space $H$ are uniformly bounded by 
$e^{2S}$.
Fix $n$ and $\omega$; write $h_n$ for $h_n(\omega)$. First we consider
the case that $h_n \geq 0$,
\[
  \begin{split}
  \Bigl(\frac{\eta^{1/2}(x/\tht_n(h_n))}{\tht_n^{1/2}(h_n)}
    &-\frac{\eta^{1/2}(x/\tht_0)}{\tht_0^{1/2}}\Bigr)^2\\
  & = 
   \frac{\eta(x/\tht_n(h_n))}{\tht_n(h_n)}\1_{[\theta_0,\theta_n(h_n)]}(x)
    +\Bigl(\frac{\eta^{1/2}(x/\tht_n(h_n))}{\tht_n^{1/2}(h_n)}
    -\frac{\eta^{1/2}(x/\tht_0)}{\tht_0^{1/2}}\Bigr)^2
      \1_{[0,\tht_0]}(x).
  \end{split}
\]
Note that the first term is bounded from above by 
$(e^{2S}/\tht_0)\1_{[\tht_0,\tht_n(M)]}(x)$. To upper bound the second term, 
we use the absolute continuity of $\eta^{1/2}$. Let 
$g(y) = (\eta(x/y)/y)^{1/2}$,
\[
  \Bigl|\frac{\eta^{1/2}(x/\tht_n(h_n))}{\tht_n^{1/2}(h_n)}
   -\frac{\eta^{1/2}(x/\tht_0)}{\tht_0^{1/2}}\Bigr|
    = \Bigl|\int_{\tht_0}^{\tht_n(h_n)}g'(y)\,dy\Bigr|
    \leq \int_{\tht_0}^{\tht_n(M)}\Bigl|g'(y)\Bigr|\, dy.
\]
By the definition of the nuisance space, for $y\in[\tht_0,\tht_n(M)]$, and 
$x \leq \tht_0$,
\[
  |g'(y)| \leq \frac{e^S}{\tht_0^{3/2}}(S+1),
\]
and then,
\[
  \Bigl(\frac{\eta^{1/2}(x/\tht_n(h_n))}{\tht_n^{1/2}(h_n)}
    -\frac{\eta^{1/2}(x/\tht_0)}{\tht_0^{1/2}}\Bigr)^2 \leq
   \frac{M^2}{n^2}\frac{e^{2S}}{\tht_0^3}(S+1)^2.
\]
Similarly for $h_n < 0$,
\[
  \begin{split}
  \Bigl(\frac{\eta^{1/2}(x/\tht_n(h_n))}{\tht_n^{1/2}(h_n)}
    &-\frac{\eta^{1/2}(x/\tht_0)}{\tht_0^{1/2}}\Bigr)^2\\
   &\leq \frac{e^{2S}}{\tht_0-M/n}\1_{[\tht_n(-M),\tht_0]}(x) 
     + \frac{M^2}{n^2}\frac{e^{2S}}{(\tht_0-M/n)^3}
      \Bigl(\frac{S\tht_0}{\tht_0-M/n}+1\Bigr)^2\1_{[0,\tht_0]}(x).
  \end{split}
\]
Combining these results, we obtain a bound for the squared Hellinger distance:
\[
    H^2(P_{\theta_n(h_n),\eta},P_{\theta_0,\eta})
    \leq \frac{Me^{2S}}{n\tht_0} + \frac{Me^{2s}}{n\tht_0-M}
    + \frac{M^2}{n^2}\frac{e^{2S}}{\tht_0^2}(S+1)^2
    + \frac{M^2}{n^2}\frac{e^{2S}\tht_0}{(\tht_0-M/n)^3}
          \Bigl(\frac{S\tht_0}{\tht_0-M/n}+1\Bigr)^2
\]  
\end{proof}
To verify condition {\em (vi)} of Corollary~\ref{cor:RateFree} we now check 
condition (\ref{eq:Condition}) of Lemma~\ref{lem:MarginalLR}.
\begin{lemma}
\label{lem:LR2}
Let $(M_n)$, $M_n\to\infty$, $M_n \leq n$ for $n \geq 1$, $M_n=o(n)$ be given. 
Then there exists a constant $C>0$ such that the condition of 
Lemma~\ref{lem:MarginalLR} is satisfied.
\end{lemma}
\begin{proof}
The proof of this lemma is similar to the proof of Lemma~\ref{lem:LR}. Therefore, 
we only note that by absolute continuity of $\eta$ we have
\[
	\frac{\eta(X_i/\tht^*)}{\tht^*} = \frac{\eta(X_i/\tht_0)}{\tht_0} 
       + \int_{\tht_0}^{\tht^*}g'(y)\, dy,
\]
where $g(y) = \eta(X_i/y)/y$, and
\[
	g'(y) = \frac{1}{y}\eta'(X_i/y)\Bigl(-\frac{X_i}{y^2}\Bigr) 
	 +\eta(X_i/y)\Bigl(-\frac{1}{y^2}\Bigr) \leq  \eta(X_i/y)\Bigl(-\frac{1}{y^2}\Bigr).
\]
\end{proof}
\whitespace

To demonstrate that priors exist such that $\scrL\subset\supp(\Pi_\scrL)$,
an explicit construction based on the distribution of Brownian sample
paths is provided in the following simplified version of Lemma~\ref{lem:Prior}. 
\begin{lemma}
\label{lem:Prior2}
Let $S > 0$ be given. Let $\{W_t:t\in[0,1]\}$ be Brownian motion on
$[0, 1]$ and let $Z$ be independent and distributed $N(0,1)$. We
define the prior $\Pi_\scrL$ on $\scrL$ as the distribution of the
process,
\[
  \score(t) = S\Psi(Z+W_t),
\]
where $\Psi:[-\infty,\infty]\rightarrow[-1,1]: x\mapsto 2\arctan (x)/\pi$.
Then $\scrL\subset\supp\bigl(\Pi_\scrL\bigr)$.
\end{lemma}

\begin{lemma}
\label{lem:IntBounds}
For every differentiable $\eta$ and $\ep > 0$ the following inequalities hold:
\[
  \eta(0)\ep - \ep\int_0^\ep|\eta'(x)|\, dx \leq \int_0^\ep \eta(x)\, dx 
  \leq \eta(0)\ep + \ep \int_0^\ep |\eta'(x)|\, dx.
\]
\end{lemma}
\begin{proof}
Integration by parts yields
\[
  \int_0^\ep\eta(x)\, dx = \eta(0)\ep +\int_0^\ep (\ep-x)\eta'(x)\, dx.
\]
Since $-\ep|\eta'(x)|\leq (\ep-x)\eta'(x) \leq \ep|\eta'(x)|$ for $x\in[0,\ep]$,
the assertion holds.
\end{proof}


\section*{Acknowledgments}
\label{sec:ack}

The authors would like to thank A.~W.~van der Vaart and J.~H.~van Zanten for
valuable discussions and suggestions. B.~J.~K. Kleijn would also like to thank
the {\it Statistics Department of Seoul National University, South Korea}
and the {\it Dipartimento di Scienze delle Decisioni, Universita Bocconi,
Milano, Italy} for their kind hospitality. B.~T. Knapik was supported by the
Netherlands Organization for Scientific Research NWO and ANR Grant `Banhdits' 
ANR--2010--BLAN--0113--03. He would also like to acknowledge the support 
of {\it CEREMADE, Universit\'e Paris-Dauphine} and {\it ENSAE--CREST}, and thank 
I.~Castillo, J.~Rousseau and A.~Tsybakov for insightful comments.


\end{document}